\documentclass[12pt]{amsart}  
\usepackage{amssymb,amsmath,amsthm,amscd}
\usepackage[all]{xy}

\usepackage{graphicx}
\usepackage{psfrag}

\usepackage[hypertex]{hyperref}

\numberwithin{equation}{section}

\newtheoremstyle{fancy1}{10pt}{10pt}{\itshape}{12pt}{\textsc\bgroup}{.\egroup}{8pt}{
}
\newtheoremstyle{fancy2}{10pt}{10pt}{}{12pt}{\itshape}{.}{8pt}{ }

\theoremstyle{fancy1}

\newtheorem{lem}[equation]{Lemma}

\newtheorem{thm}[equation]{Theorem}

\newtheorem*{main*}{Theorem}

\newtheorem*{cor*}{Corollary}

\newtheorem*{problem*}{Problem}

\theoremstyle{fancy2}

\newtheorem*{rem*}{Remark}

\newcommand{\cref}[1]{Corollary~\ref{#1}}

\newcommand{\lref}[1]{Lemma~\ref{#1}}

\newcommand{\tref}[1]{Theorem~\ref{#1}}

\newcommand{\gl}{\lambda}
\newcommand{\gt}{\theta}
\newcommand{\e}{\epsilon}

\newcommand{\RP}{\mathbb{R\mkern1mu P}}
\newcommand{\CP}{\mathbb{C\mkern1mu P}}

\newcommand{\Sph}{\mathbb{S}}

\newcommand{\C}{{\mathbb{C}}}
\newcommand{\R}{{\mathbb{R}}}
\newcommand{\Z}{{\mathbb{Z}}}

\newcommand{\QH}{{\mathbb{H}}}

\newcommand{\G}{\ensuremath{\operatorname{G}}}

\newcommand{\SO}{\ensuremath{\operatorname{SO}}}

\renewcommand{\O}{\ensuremath{\operatorname{O}}}
\newcommand{\Sp}{\ensuremath{\operatorname{Sp}}}
\newcommand{\U}{\ensuremath{\operatorname{U}}}
\newcommand{\SU}{\ensuremath{\operatorname{SU}}}

\newcommand{\Pin}{\ensuremath{\operatorname{Pin}}}

\renewcommand{\S}{\ensuremath{\operatorname{S}}}

\newcommand{\fg}{{\mathfrak{g}}}

\newcommand{\fh}{{\mathfrak{h}}}
\newcommand{\fm}{{\mathfrak{m}}}

\newcommand{\fso}{{\mathfrak{so}}}
\newcommand{\fsu}{{\mathfrak{su}}}

\newcommand{\pro}[2]{\langle #1 , #2 \rangle}

\newcommand{\orth}[2]{\text{S(O(#1)O(#2))}}
\def\con#1=#2(#3){#1 \equiv #2 \bmod{#3}}

\newcommand{\ml}{\langle}                     
\newcommand{\mr}{\rangle}                    

\newcommand{\tr}{\ensuremath{\operatorname{tr}}}
\newcommand{\diag}{\ensuremath{\operatorname{diag}}}

\renewcommand{\Im}{\ensuremath{\operatorname{Im}}}
\newcommand{\Ad}{\ensuremath{\operatorname{Ad}}}

\renewcommand{\sec}{\ensuremath{\operatorname{sec}}}

 \DeclareMathOperator{\Real}{Re}

\DeclareMathOperator{\Id}{Id} 
 
\DeclareMathOperator{\spam}{span}

\newcommand{\Kpmo}{K_{\scriptscriptstyle{0}}^{\scriptscriptstyle{\pm}}}
\newcommand{\Kpo}{K_{\scriptscriptstyle{0}}^{\scriptscriptstyle{+}}}
\newcommand{\Kmo}{K_{\scriptscriptstyle{0}}^{\scriptscriptstyle{-}}}
\newcommand{\Kpm}{K^{\scriptscriptstyle{\pm}}}
\newcommand{\Kp}{K^{\scriptscriptstyle{+}}}
\newcommand{\Km}{K^{\scriptscriptstyle{-}}}
\newcommand{\Ko}{K_{\scriptscriptstyle{0}}}

\newcommand{\subo}{_{\scriptscriptstyle{0}}}

\newcommand{\no}{\noindent}
\newcommand{\co}{{cohomogeneity}}
\newcommand{\coo}{{cohomogeneity one}}
\newcommand{\com}{{cohomogeneity one manifold}}
\newcommand{\coms}{{cohomogeneity one manifolds}}
\newcommand{\coa}{{cohomogeneity one action}}
\newcommand{\coas}{{cohomogeneity one actions}}

\newcommand{\nnc}{non-negative curvature}

\newcommand{\pc}{positive curvature}
\newcommand{\psc}{positive sectional curvature}
\newcommand{\pcu}{positively curved}

\begin{document}

\title{On the geometry of cohomogeneity one manifolds with positive curvature}

\author{Wolfgang Ziller}
\address{University of Pennsylvania\\
   Philadelphia, PA 19104}
\email{wziller@math.upenn.edu}

\thanks{The  author was supported by  a grant from the
National Science Foundation, by IMPA in Rio de Janeiro and the Max
Planck Institute in Bonn, and would like to thank the Institutes for
their hospitality.}

\maketitle

There are very few known examples of manifolds with positive
sectional curvature. Apart from the compact rank one symmetric
spaces, they exist only in dimensions 24 and below  and are all
obtained as quotients of a compact Lie group equipped with a
biinvariant metric under an isometric group action.
 They
consist of certain homogeneous spaces in dimensions $6,7,12,13$ and
$24$ due to Berger \cite{Be}, Wallach \cite{Wa}, and Aloff-Wallach
\cite{AW}, and of biquotients in dimensions $6,7$ and $13$ due to
Eschenburg \cite{E1},\cite{E2} and Bazaikin \cite{Ba}.

When trying to find new examples, it is natural to search among
manifolds with large isometry group, a program initiated by K.Grove
in the 90's, see \cite{Wi} for a recent survey.  Homogeneous spaces
with positive curvature were classified in \cite{Wa},\cite{BB} in
the 70's. The next natural case to study is therefore manifolds on
which a group acts isometrically with one dimensional quotient, so
called \com s.  L.Verdiani \cite{V1,V2} showed  that in even
dimensions, positively curved \com s are equivariantly diffeomorphic
to an isometric action on a rank one symmetric space. In odd
dimensions K.Grove and the author observed in 1998 that there are
infinite families among the known non-symmetric positively curved
manifolds which admit isometric \coa s, and suggested a family  of
potential 7 dimensional candidates $P_k$. In \cite{GWZ} a
classification in odd dimensions was carried out and another family
$Q_k$ and an isolated manifold $R$ emerged in dimension $7$. It is
not yet known whether these manifolds admit a \coo\ metric with
positive curvature, although they all admit one with \nnc\ as a
consequence of the main result in \cite{GZ}.

In \cite{GWZ} the authors also discovered an intriguing connection
that the manifolds $P_k$ and $Q_k$ have with a family of self dual
Einstein orbifold metrics constructed by Hitchin \cite{Hi1} on
$\Sph^4$. They naturally give rise to 3-Sasakian metrics on $P_k$
and $Q_k$, which by definition have lots of positive curvature
already.

The purpose of this survey is three fold. In Section 2 we study the
positively curved \coo\ metrics on known examples with positive
curvature including the explicit  functions that define the metric.
In Section 3 we describe the classification theorem in \cite{GWZ}.
 It
is remarkable that among 7-manifolds where $G=\S^3\times\S^3$ acts
by \coo, one has the known positively curved Eschenburg spaces
$E_p$, the Berger space $B^7$, the Aloff-Wallach space $W^7$, and
the sphere $\Sph^7$, and that the candidates $P_k$, $Q_k$ and $R$
all carry such an action as well. We thus carry out the proof in
this most intriguing case where $G=\S^3\times\S^3$ acts by \coo\ on
a compact 7-dimensional simply connected manifold.
 In Section 4 we describe the relationship to
Hitchin's self dual Einstein metrics. We also discuss some curvature
properties of these Einstein metrics and the metrics they define on
$P_k$ and $Q_k$.  The behavior of these metrics, as well as the
known metrics with positive curvature, are illustrated in a series
of pictures.

\bigskip

\section{Preliminaries}
\label{general}

\smallskip

In this section we  discuss the basic structure of \coas\ and the
significance of the Weyl group. For more details we refer the reader
to  \cite{AA,Br,GZ,Mo}. We assume from now on that the manifold $M$
and the group $G$ that acts on $M$ are compact and will only
consider the most interesting case, where $M/G = I=[0,L]$. If
$\pi\colon M\to M/G$ is the orbit projection, the inverse images of
the interior points are the regular orbits and $B_- = \pi^{-1}{(0)}$
and $B_+ = \pi^{-1}{(L)}$ are the two non-regular orbits.  Choose a
point $x_- \in B_-$ and let $\gamma : [0,L] \to M$ be a minimal
geodesic from $B_-$ to $B_+$, parameterized by arc length, which we
can assume starts at $x_-$. The geodesic is orthogonal to $B_-$ and
hence to all orbits. Define $x_+=\gamma(L)\in B_+\; $,
$x_0=\gamma(\frac L 2)$ and let $\Kpm = G_{x_\pm}$ be the isotropy
groups at $x_\pm$ and $H = G_{x_0} = G_{\gamma(t)}$, $ 0<t<L $, the
principal isotropy group. Thus $B_\pm =G\cdot x_\pm= G/\Kpm$ and
$G\cdot\gamma(t)=G/H$ for $ 0<t<L $. For simplicity we denote the
tangent space of $B_\pm$ at $x_\pm$ by $T_\pm$ and its normal space
by $T_\pm^{\perp}$.

By the slice theorem, we have the following description of the
tubular neighborhoods $D(B_-)=\pi^{-1}([0,\frac L 2])$ and
$D(B_+)=\pi^{-1}([\frac L 2,L])$ of the nonprincipal orbits:
\begin{equation*}
\label{discbundle} D(B_{\pm}) = G\times_{K_{\pm}}D^{\ell_{\pm}} ,
\end{equation*}
\no where $D^{\ell_{\pm}}$ are  disks of radius $\frac L 2$ in
$T^{\perp}_\pm$. Here the action of $K_{\pm} $ on $G\times
D^{\ell_{\pm}}$ is given by $k\star (g,p)=(gk^{-1},kp)$ where $k$
acts on $D^{\ell_{\pm}}$ via the slice representation, i.e., the
restriction of the isotropy representation to $T_\pm^\perp$. Hence
we have the decomposition
\begin{equation*}
\label{decomp} M=D(B_-) \textstyle{\: \cup}_E \: D(B_+) \; ,
\end{equation*}
\no where $E = G\cdot x_0 = G/H$ is a principal orbit which is
canonically identified with the boundaries $\partial D(B_\pm) =
G\times_{K_{\pm}}\Sph^{\ell_{\pm}-1}$, via the maps $g\cdot H\to
[g,\dot{\gamma}(0)]$ respectively $g\cdot H\to
[g,-\dot{\gamma}(L)]$. Note also that $\partial D^{\ell_{\pm}} =
\Sph^{\ell_{\pm}-1} = \Kpm /H$ since the boundary of the tubular
neighborhoods must be a $G$ orbit and hence $\partial
D^{\ell_{\pm}}$ is a $\Kpm$ orbit.

All in all we see that we can recover $M$ from $G$ and the subgroups
$H$ and $\Kpm$.  We caution though that the isotropy types, i.e.,
the conjugacy classes of the isotropy groups $\Kpm $ and $ H$ do not
determine $M$. The isotropy groups  depend on the  choice of a
minimal geodesic between the two non-regular orbits, and thus on the
metric as well. A different choice of a minimal geodesic corresponds
to conjugating all isotropy groups by an element of $G$. A change of
the metric corresponds to changing $\Kp$ to $n\Kp n^{-1}$ for some
$n\in N(H)_{\subo}$, the identity component of the normalizer (cf.
\cite{AA,GWZ}).

\smallskip

An important fact about \coas\ is that there is a converse to the
above construction. Suppose $G$ is a compact Lie group with
subgroups $H\subset \Kpm \subset G$ and assume furthermore that
$\Kpm/H = \Sph^{\ell_{\pm}-1}$ are spheres. We sometimes denote this
situation by $H\subset \{\Km,\Kp\}\subset G$ and call it a group
diagram.  It is well known that a transitive action of a compact Lie
group $K$ on a sphere $\Sph^{\ell-1}$ is conjugate to a linear
action. We can thus assume that $\Kpm$ acts linearly on
$\Sph^{\ell_\pm}$ with isotropy group the chosen subgroup $H\subset
\Kpm $ at some point $p_\pm\in \Sph^{\ell_\pm-1}$. It hence extends
to a linear action on the bounding disk $D^{\ell_{\pm}}$ and we can
thus define a manifold
\begin{equation*}
\label{manifold} M = G\times_{\Km}D^{\ell_{-}} \cup_{\ G/H}
G\times_{\Kp}D^{\ell_{+}} ,
\end{equation*}
\no where we glue the two boundaries by sending $[g,p_-]$ to
$[g,p_+]$. The group
 $G$ acts on $M$   via $g\star [g',p]=[gg',p]$ on each half and one easily
checks that the gluing is $G$-equivariant, and that the action  has
isotropy groups $\Kpm $ at $[e,0]$ and $ H$ at $[e,p_\pm]$. One may
also choose an equivariant map $G/H\to G/H$ to glue the two
boundaries together. But such equivariant maps are given by $gH\to
gnH$ for some $n\in N(H)$ and the new manifold is alternatively
obtained by replacing $\Kp$ with $n\Kp n^{-1}$ in the group diagram.
But we caution that this new manifold may not be equivariantly
diffeomorphic to the old one if $n$ does not lie in the identity
component of $N(H)$.

\smallskip

Another important ingredient for understanding the geometry of a
\com\ is given by the Weyl group. The Weyl group $W$ of the action
is by definition the stabilizer of the geodesic $\gamma$ modulo its
kernel, which by  construction is equal to $H$. It is easy to see
(cf. \cite{AA}) that $W$ is a dihedral subgroup of $N(H)/H$ with
$M/G = \Im(\gamma)/W$, and is generated by involutions $w_{\pm} \in
W$ with $w_-(\gamma(t))=\gamma(-t)$ and
$w_+(\gamma(-t+L))=\gamma(t+L)$. Thus $w_+w_-$ is a translation by
$2L$, and has order $|W|/2$ when $W$ is finite. The involutions
$w_\pm$ can  be represented by the unique
 element $a\in\Kpmo$ mod $H$ with $av=-v$, where $\Kpm_v=H$.  If
 $\ell_\pm=1$, they are also the unique
 element $a\in\Kpmo$ mod $H$ such that
$a^2$ but not $a$ lies in $H$.
  For simplicity we denote such representatives
$a\in\Kpmo$ again by $w_\pm$.

Note that $W$ is finite if and only if $\gamma$ is a closed
geodesic, and in that case the order $|W|$ is the number of minimal
geodesic segments intersecting the regular part. In \cite{GWZ} it
was shown that a \com\ with an invariant metric of positive
curvature necessarily has finite Weyl group. Note also that any
non-principal isotropy group along $\gamma$ is of the form $w \Kpm
w^{-1}$ for some $w \in N(H)$ representing an element of $W$, and
that the isotropy types $\Kpm$ alternate along $\gamma$.

\smallskip

We now discuss how to describe \coo\ metrics on $M$. For $0<t<L$,
 $\gamma(t)$ is a regular
 point  with constant isotropy group $H$ and the metric on the
 principal
 orbits $G \gamma(t)=G/H$ is a family of homogeneous metrics
 $g_t$.
 Thus on the regular part
 the metric is determined by
$$g_{\gamma(t)}=d\,t^2+g_t,$$
and since the regular points are dense it also describes the metric
on $M$. Using a fixed biinvariant inner product $Q$ on $\fg$ we
define the $Q$-orthogonal splitting $\fg=\fh\oplus\fm$ which thus
satisfies $\Ad(H)(\fm)\subset\fm$. We  identify the tangent space to
$G/H$ at $\gamma(t), t\in(0,L)$ with $\fm$ via action fields:
$X\in\fm\to X^*(\gamma(t))$, which also identifies the isotropy
representation with the action of $\Ad(H)_{|\fm}$.  We can choose a
$Q$-orthogonal decomposition $\fm=\fm_1 + \cdots + \fm_k$ of $\fm$
into $\Ad(H)$ invariant irreducible subspaces and thus
$g_t{|\fm_i}=f_i(t)Q{|\fm_i}$ for some functions $f_1,\dots ,f_k$.
If the modules $\fm_i$ are inequivalent to each other, they are
automatically orthogonal and the functions $f_i$ describe the metric
completely. In positive curvature, it typically happens, as we will
see, that the modules are not orthogonal to each other and further
functions are necessary to describe their inner products. In order
for the metric on $M$ to be smooth, these functions must satisfy
certain smoothness conditions at the endpoints $t=0$ and $t=L$,
which in general can be complicated.  The action of $w_\pm$ on
$T_{x_\pm}M$ (well determined only up to $\Ad(H)$), preserves
$T_\pm$ and $T^{\perp}_\pm$
 and the action on $T_\pm$, given by
$\Ad(w_\pm)$, relates the functions describing the metric:
$(\Ad(w_-)(X))^*(\gamma(t))=X^*(\gamma(-t))$. If, e.g.,
$\Ad(w_-)(\fm_i)\subset \fm_i$, the function $f_i$ must be even, and
if $\Ad(w_-)(\fm_i)\subset\fm_j$, then $f_i(t)=f_j(-t)$. In fact
most, but not all, of the smoothness conditions at the endpoints can
be explained in this fashion by the action of the Weyl group.

\bigskip

\section{Known examples of cohomogeneity one manifolds with positive curvature}

\bigskip

In this section we describe the \coo\ actions on the known \com s
with \pc\, which were discovered by K.Grove and the author in 1998.
Apart from a \coo\ action by $\SU(4)$ on the infinite family of $13$
dimensional Bazaikin spaces $B^{13}_p$, which we will not discuss in
this survey, they are all \coo\ under an action of $\S^3\times\S^3$
or one of its finite quotients. We start with the well known action
of $\SO(3)$ on $\Sph^4$ and $\SO(3)$ on $\CP^2$ since they are
important in understanding the remaining examples and determine much
of their geometry. We then study the action of $\SO(4)$ on $\Sph^7$,
and of $\SO(4)$ on the Berger space $B^7=\SO(5)/\SO(3)$. This latter
action was also discovered by Verdiani-Podesta in \cite{PV2}. Of a
different nature is the action of $\SU(2)\times\SO(3)$ on the
infinite family of Eschenburg biquotients $E_p$, which contains as a
special case a homogeneous Aloff-Wallach space. We finish with a
second \coo\ action on the same Aloff-Wallach space which shares
some features of both actions. To distinguish them, we denote the
first one by $W^7_{(1)}$ and the second by  $W^7_{(2)}$. For a
survey of the known examples of \pc, see \cite{Zi2}.

All actions described here are by groups locally isomorphic to
$\S^3$ or $\S^3\times\S^3$. For comparison we will describe them
ineffectively so that $G=\S^3$ or $G=\S^3\times\S^3$ acts on the
manifold. The effective version of the action, which we denote by
$\bar{G}$, is obtained by dividing $G$ by the ineffective kernel,
which is the intersection of the center of $G$ with the principal
isotropy group $H$.

\bigskip

\begin{center} $M=\Sph^4$ or $\CP^2$ \text{ with } $\bar{G}=\SO(3)$.
\end{center}

\bigskip

We begin
 by describing the well known \co\ one action by \SO(3) on
$\Sph^4$.
 Let $V$ be the 5-dimensional vector space of real $ 3\times 3$
matrices with $A=A^t , \tr (A)=0$ and with inner product $\pro{A}{B}
= \tr AB$. The group $\SO(3)$ acts on $V$ via conjugation $g\cdot
A=gAg^{-1}$ preserving the inner product and hence acts on
$\Sph^4(1) \subset V$. Every point in $\Sph^4(1)$ is conjugate to a
matrix in the great circle $F = \{\diag(\gl_1,\gl_2,\gl_3) \mid
 \sum \gl_i =0 , \sum \gl_i^2 = 1\}$ and hence the quotient space is
 one dimensional.
 For the purpose of computations, we choose an orthonormal basis
 \begin{equation}\label{basis}
 e_1=\frac{\diag(1,1,-2)}{\sqrt{6}}\; ,\; e_2=\frac{\diag(1,-1,0)}{\sqrt{2}}
 \; ,\;
 e_3=\frac{S_{12}}{\sqrt{2}}\; ,\; e_4=\frac{S_{13}}{\sqrt{2}}\; ,\;
 e_5=\frac{S_{23}}{\sqrt{2}},
 \end{equation}
where $S_{ij}$ is a symmetric matrix with a one in entries $ij$ and
$ji$ and 0 everywhere else. If we choose $x_-=e_1$, then clearly
$\Km=\S(\O(2)\O(1))$ and the orbit $G/\Km$ is the set of all
symmetric matrices with 2 equal positive eigenvalues. Furthermore,
$T_-=\spam(e_4,e_5)$ with $T_-^\perp=\spam(e_2,e_3)$ and thus
$\gamma(t)=\cos(t)e_1+\sin(t)e_2$ is a geodesic orthogonal to $B_-$
and hence to all orbits. Clearly
$x_+=\gamma(\pi/3)=\diag(2,-1,-1)/\sqrt{6}$ is the first point along
the geodesic $\gamma$ which lies on the second singular orbit,
consisting of the set of all symmetric matrices with 2 equal
negative eigenvalues. Thus $L=\pi/3$ and $\Kp=\S(\O(1)\O(2))$. For
$\gamma(t)$ with $0<t<\frac \pi 3$, all eigenvalues $\lambda_i$ are
distinct and hence the principal  isotropy group is
$H=\text{S(O(1)O(1)O(1))} = \Z_2\times \Z_2$.

\smallskip

 If we denote by
$E_{ij}$ the usual basis of the set of skew symmetric matrices, the
above action of $\SO(3)$ on $\Sph^4(1)$ induces 3 action fields
$E^*_{ij}$. A computation shows that:
$$E^*_{12}=2\sin(t)e_3\; , \;
E^*_{23}=(\sqrt{3}\cos(t)-\sin(t))e_5\; , \;
E^*_{13}=(\sqrt{3}\cos(t)+\sin(t))e_4.$$  For the  functions $f_1=|E^*_{12}|^2 , \;
f_2=|E^*_{23}|^2, \; f_3=|E^*_{13}|^2$, which describe the metric,
we thus obtain:
$$f_1=4\sin^2(t) , \; f_2=(\sqrt{3}\cos(t)-\sin(t))^2 , \;
f_3=(\sqrt{3}\cos(t)+\sin(t))^2,$$ and all other inner products are
$0$.

\smallskip

For later purposes it will be convenient to  lift the isotropy
groups into $\S^3$ under the two-fold cover $\S^3=\Sp(1) \to
\SO(3)$, given by conjugation on $\QH$, which sends $q\in \Sp(1)$
into a rotation in the 2-plane $\Im(q)^\perp\subset \Im(\QH)$ with
angle $2\gt$, where $\gt$ is the angle between $q$ and $1$ in
 $\S^3$.
After renumbering the coordinates, the group $\Km$ lifts to
$\Pin(2)=\{e^{i\gt}\mid\gt\in\R\} \cup \{je^{i\gt}\mid\gt\in\R\}$
which we abbreviate to $e^{i\gt} \cup je^{i\gt}$. Similarly, $\Kp$
lifts to $\Pin(2)=e^{j\gt} \cup ie^{j\gt}$, and $H=
$S(O(1)O(1)O(1))$\subset \SO(3)$ lifts to the
 quaternion group $Q=\{\pm 1, \pm i,\pm j,\pm k\}$.
Thus the \com\ $\Sph^4$ can also be represented by the group diagram
$$ Q \subset\{e^{i\gt} \cup je^{i\gt} ,e^{j\gt} \cup ie^{j\gt}\}\subset\S^3.$$

We will now discuss the Weyl symmetry using this group picture.
Clearly $w_-=e^{i\frac \pi 4}$ since $e^{i\frac \pi 2}=i$ lies in
$H$ and similarly $w_+=e^{j\frac \pi 4}$. Thus $a=w_+w_-=\frac 1 2
(1+i+j+k)$ represents a translation by $2L$ along the geodesic and
since $a^3=-1\in H$, the Weyl group is $W=D_3$. This is consistent
with the fact that the angle between $x_-$ and $x_+$ is $\pi/3$ and
hence $\gamma(t)$ intersects the regular part in $6=|W|$ components.
Notice now that $aia^{-1}=j , aja^{-1}=k$, and $ aka^{-1}=i$. Thus
$f_1(t+2L)=f_2(t)$ and $f_1(t+4L)=f_3(t)$. By applying only $w_+$ at
$t=L$, we also obtain $f_3(t)=f_1(-t+2L)$. Thus the function
$f_1(t)$ on the interval $[0,3L]$ determines the full geometry of
the \com:
\begin{equation}\label{Weyl}
f_2(t)=f_1(t+2L)\; , \; f_3(t)=f_1(t+4L)=f_1(-t+2L)\quad,\quad 0<t<L
 \end{equation}

\smallskip

 There is a related \coa\ by $\SO(3)\subset\SU(3)$ on $\CP^2$ which
 has singular orbits the real points $B_-=\RP^2\subset\CP^2$ and the
 quadric $B_+=\Sph^2=\{[z_0,z_1,z_2]\mid \sum z_i^2=0\}$.
Here we use the metric on $\CP^2$ induced by the biinvariant metric
$\ml A,B\mr = -\frac 1 2 \Real \tr AB$ on $\SU(3)$, which has
curvature $1\le \sec\le 4$.
  One
 easily shows
that the unit speed geodesic $\gamma(t)$, given in homogeneous
coordinates by $[(\cos(t),i\sin(t),0)]$, is orthogonal to all orbits
and that the isotropy group at $x_-=\gamma(0)$ is $\Km=\orth{1}{2}$,
at $x_+=\gamma(\pi/4)$ is $\Kp=\SO(2)$, embedded  in the first two
coordinates, and that $H=G_{\gamma(t)}=\Z_2=
\langle\diag(-1,-1,1)\rangle$ for $0<t<\pi/4$. Hence $L=\pi/4$ and
the group diagram, lifted to $\S^3$, is  given by:
$$ \{\pm 1,\pm j\} \subset\{e^{i\gt}\cup je^{i\gt} ,e^{j\gt} \}\subset\S^3.$$
Thus  a projection along the orbits gives rise to a two fold
branched
 cover $\CP^2\to\Sph^4$ with branching locus the singular orbit
 $\RP^2=\S^3/(e^{i\gt} \cup je^{i\gt}) $. One easily shows that
 the functions describing
 the metric are given by
 $$f_1=\sin^2(t) , \; f_2=\cos^2(2t) , \;
f_3=\cos^2(t).$$ The Weyl group symmetry changes since $w_-=i $ but
$w_+=e^{j\frac \pi 4}$ and hence $W=D_2$. The functions are thus
related by \begin{equation}\label{Weyl2}
f_3(t)=f_1(t+2L)=f_1(-t+2L)\; , \; f_2(t)=f_2(-t)=f_2(-t+2L).
\end{equation}
Hence in this case the functions $f_1$ and $f_2$ on $[0,2L]$
determines the full geometry of the \com.

\smallskip

We finally mention the \coo\ action by $\SU(2)$ on $\CP^2$ which
fixes a point $p_0$. The second singular orbit is then the cut locus
of $p_0$ and hence $L=\pi/2$ in this case.

\bigskip


\begin{center} $M=\Sph^7$ \text{ with } $\bar{G}=\SO(4)$.
\end{center}

\bigskip

The action of $\SO(4)$ on $\R^8$ which induces a \coo\ action on
$\Sph^7$ is given by the isotropy representation of the rank 2
symmetric space $\G_2/\SO(4)$ on its tangent space. As a complex
representation, it is the representation of $\SU(2)\times \SU(2)$
obtained by taking the tensor product of the unique 2-dimensional
irreducible representation of $\SU(2)$ with the 4-dimensional one on
the second factor. Thus the first $\SU(2)$ factor acts as the Hopf
action on $\Sph^7$ and the second factor induces an action by
$\SO(4)/\SU(2)\simeq \SO(3)$ on $\Sph^7(1)/\SU(2)=\Sph^4(\frac 1
2)$. This action of $\SO(3)$ on $\R^5$ is irreducible since the
representation of $\SU(2)\times \SU(2)$ on $\R^8$ is irreducible.
Thus it agrees with the \coo\ action of $\SO(3)$ in the previous
example and hence the $\bar{G}$ action on $\Sph^7$ is \co\ as well.
Since $\SU(2)$ acts freely, this implies in particular that both
actions have isomorphic isotropy groups, i.e. $\Kpm\simeq\O(2)$ and
$H\simeq\Z_2\oplus\Z_2$ considered as subgroups of $\SO(4)$. We now
need to determine their explicit embeddings into $\SO(4)$
respectively $\S^3\times\S^3$.

 If we let $V_k$ be the vector space of homogeneous
polynomials of degree $k$ in two complex variables $z,w$, then
$\SU(2)$ acting on vectors $(z,w)$ via matrix multiplication,
 induces an irreducible representation on $V_k$ of (complex)  dimension $k+1$
 and preserves the inner product which makes $z^mw^n$ into an orthogonal basis
  with  $|z^mw^n|^2=m!n!$.
 The isotropy representation (complexified) is thus
 $V_1\otimes V_3$. The map $(z,w)\to (w,-z)$, extended to be a
 complex antilinear map $J_k\colon V_k\to V_k$, satisfies
 $J_k^2=(-1)^k\Id$ and hence $(J_1\otimes J_3)^2=\Id$. Thus $J_1\otimes J_3$ induces a real
 structure on $\C^8$ and hence its $+1$ eigenspace $W$ is invariant
 under the action of $G=\SU(2)\times\SU(2)$, and is spanned by:
 \begin{align*}&xz^3+yw^3\; , \;   i(xz^3-yw^3) \; , \; xzw^2+ywz^2\; , \;
 i(xzw^2-ywz^2),\\
 &yz^3-xw^3\; , \;   i(yz^3+xw^3) \; , \; xz^2w-yw^2z\; , \;
  i(xz^2w+yw^2z),
  \end{align*}
  which is our desired representation of $G$ on $\R^8$.

  Now let $\Delta Q$ be the diagonal embedding of the quaternion
  group into $\SU(2)\times\SU(2)=\S^3\times\S^3 $ , i.e.
  $\Delta Q=\{\pm(1,1),\pm(i,i),\pm(j,j),\pm(k,k)\}$. Here we identify
  $a+bj\in \S^3$ with ${\tiny \left(
                 \begin{array}{cc}
                   a    & b \\
                   -\bar{b} & \bar{a} \\
                 \end{array}
               \right)\in\SU(2)  }$.
   One easily
  checks that    $\Delta Q$ fixes the two plane spanned by the
  orthonormal vectors $a=(xz^3+yw^3)/2\sqrt{3}\; , \;
  b=(xzw^2+ywz^2)/2$.
    By the above, this is then also the principal isotropy
  group $H$ since the image of $\Delta Q$ in $\SO(4)$ is $\Z_2\oplus\Z_2$.
   The great circle $\gamma(t)=\cos(t)a+\sin(t)b$  in this
  2-plane meets all orbits orthogonally since it agrees with the
  fixed point set of $H$ on $\Sph^7$. Now one easily checks that
  $x_-=\gamma(0)$ is fixed by the circle $(e^{-3it},e^{it})\subset
  \S^3\times\S^3 $ and hence $\Km=(e^{-3it},e^{it}) \cdot H=
  (e^{-3it},e^{it})\cup (j,j)\cdot(e^{-3it},e^{it})$. The first
  singular point along $\gamma$ occurs at $x_+=\gamma(\pi/6)$ since
  the
projection of $\gamma$ is a normal geodesic in the \com\
$\Sph^4(\frac 1 2 )$. Thus $L=\pi/6$ and a computation shows that
$\gamma(\pi/6)$ is fixed by the circle $(e^{jt},e^{jt})
  $ and hence $\Kp=
  (e^{jt},e^{jt})\cup (i,i)\cdot(e^{jt},e^{jt})$.    Thus the
  group picture is given by
  $$
   H=\Delta Q \subset\{ (e^{-3it},e^{it})\cdot H \;
   ,\;
   (e^{jt},e^{jt})\cdot H
   \}\subset\S^3\times\S^3.
  $$

  We identify the Lie algebra $\fg$ with $\Im\QH\oplus\Im\QH$ and
  will use the basis $X_1=(i,0)\; , \;X_2=(j,0)\; , \;X_3=(k,0)$ and
$Y_1=(0,i)\; , \;Y_2=(0,j)\; , \;Y_3=(0,k)$ of $\fg$ and define
$f_i(t)=|X_i^*(\gamma(t))|^2\ , g_i(t)=|Y_i^*(\gamma(t))|^2 ,
h_i(t)=\langle X_i^*(\gamma(t)),Y_i^*(\gamma(t))\rangle$. Using the
above action of $\S^3\times\S^3$ applied to $\gamma(t)$, we obtain
the action fields
\begin{align*}
X_1^*(\gamma(t))&=i\cos(t)\frac{xz^3-yw^3}{2\sqrt{3}}
+i\sin(t)\frac{xzw^2-ywz^2}{2},\\
Y_1^*(\gamma(t))&=i\cos(t)\frac{3xz^3-3yw^3}{2\sqrt{3}}
+i\sin(t)\frac{-xzw^2+ywz^2}{2} ,\end{align*} and thus
$$f_1=1\; , \; g_1=8\cos^2(t)+1\; , \;
h_1=4\cos^2(t)-1.$$ As in the case of $\Sph^4$, the remaining
functions  are now determined via Weyl symmetry. We have
$w_-=e^{i\frac \pi 4}(-1,1)$
 and  $w_+=e^{j\frac
\pi 4}(1,1)$ and thus $a=w_+w_-=\frac 1 2 (1+i+j+k)(-1,1)$ with
$a^3=(-1,1)$. Hence $W=D_6$ corresponding to the fact that $\gamma$
meets $B_+$ at $t=\pi/6$ and hence intersects the regular part in 12
components. Conjugation with $a$ behaves as in the case of $\Sph^4$
on each component and hence $f_i$, as well $g_i$ and $h_i$, satisfy
the symmetry relations \eqref{Weyl}. Finally, we observe that all
remaining inner products are necessarily $0$ since the actions of
the isotropy group $\Delta Q$ on the 3 subspaces $\spam\{ X_i,Y_i\},
i=1,2,3$ are inequivalent to each other.

\bigskip

\begin{center} $M=B^7$ \text{ with } $\bar{G}=\SO(4)$.
\end{center}

\bigskip

As we saw in our first example, $\SO(3)$ acts orthogonally on the
vector space $V$, consisting of the set of traceless symmetric
$3\times 3$ matrices, via conjugation and hence isometrically on
$\Sph^4$. This gives  rise to an embedding $\phi\colon
\SO(3)\to\SO(5)$ and  defines a homogeneous space
$B^7=\SO(5)/\SO(3)$, also known as the Berger space. Berger showed
in \cite{Be} that a biinvariant metric on $\SO(5)$ induces a metric
on $B^7$ with \psc.

The subgroup $\SO(4)\subset\SO(5)$ acts on $B_7$
 via left multiplication and we claim it is \coo.
 Using the basis \eqref{basis} from Example 1, we let
$\SO(4)=\SO(5)_{e_1}$ be the subgroup fixing $e_1$.
 The isotropy groups
are then given by $\SO(4)_{g\SO(3)}=\SO(4)\cap
g\SO(3)g^{-1}=g(\SO(3)_{g^{-1}e_1})g^{-1}$. Hence it follows from
our first example that $\Kpm\simeq\O(2)$ and $H\simeq\Z_2\oplus\Z_2$
and thus the action is \coo. We now need to compute the explicit
embeddings of $\Kpm$ in $\SO(4)$ respectively $\S^3\times\S^3$.

To avoid confusion, we let $E_{ij}$ be the basis of skew symmetric
matrices in $\SO(3)$ and $F_{ij}$ the one in $\SO(5)$. If we set
$\phi_*(E_{12})=H_1\; , \; \phi_*(E_{23})=H_2\; ,
\;\phi_*(E_{13})=H_3$, one easily shows, using the explicit
description of the action of $\SO(3)$ on $V$,  that
$$H_1=2F_{23}+F_{45} \; , \;
H_2=F_{34}+\sqrt 3 F_{15}-F_{25}\; , \; H_3=F_{35}+\sqrt 3
F_{14}+F_{24}.$$ Thus $H_i$ defines an orthogonal basis of the Lie algebra of
$\phi(\SO(3))$ with $|H_i|^2=5$.

 For the point $x_-$ we choose $x_-=e\cdot
\SO(3)$,  the identity coset in $\SO(5)/\SO(3)$. The group
$\Kmo\simeq\SO(2)$ then acts by rotation with angle $2\theta$ in the
$e_2,e_3$ plane and angle $\theta$ in the $e_4,e_5$ plane since
$\phi_*(E_{12})=2F_{23}+F_{45}$, which determines its embedding into
$\SO(4)$. For $\Kp$ we need to follow a normal geodesic. Clearly,
$F_{12}$ and $ F_{13}$ are orthogonal to $H_i$ and  the orbit of
$\SO(4)$ and hence lie in the normal space of $B_-$. In $B^7$, being
normal homogeneous, a geodesic is the image of a one parameter group
with initial vector orthogonal to $H_i$. Thus   we can let
$\gamma(t)=\exp(tF_{12})\cdot\SO(3)=(\cos(t)e_1+\sin(t)e_2)\cdot\SO(3)$
be the geodesic orthogonal to all orbits. From Example 1 it follows
that the isotropy at $\gamma(t)$ is isomorphic to $\Z_2\oplus\Z_2$
for $0<t<\pi/3$ and to $\O(2)$ at $\gamma(\pi/3)$ and thus
$L=\pi/3$. If $g=\exp(\frac \pi 3 F_{12})$, we have $g^{-1}e_1=\frac
1 2 e_1+\frac{ \sqrt{3}}{ 2} e_2=\diag(2,-1,-1)/\sqrt{6}$ and hence
$\Kp=g(\S(\O(1)\O(2))g^{-1}$. From the embedding $\phi$ it is clear
that $\SO(2)\subset \S(\O(1)\O(2)\subset\SO(3)$ fixes $g^{-1}e_1$
and rotates by $\theta$ in the $e_3,e_4$ plane and by $2\theta$ in
the plane spanned by $\frac{ \sqrt{3}}{ 2} e_1-\frac 1 2 e_2$ and
$e_5$. Conjugating with $g$ gives a rotation that fixes $e_1$,
rotates by $\theta$ in the $e_3,e_4$ plane and by $-2\theta$ in the
$e_2,e_5$ plane.

We now lift these groups into $G=\S^3\times\S^3$ using the
identification $e_2\leftrightarrow 1\; , \; e_3\leftrightarrow i\; ,
\; e_4\leftrightarrow j\; , \; e_5\leftrightarrow k$ and the 2-fold
cover $\S^3\times\S^3\to\SO(4)$ given by left and right
multiplication of unit quaternions. It sends $X_1=(i,0)\to
F_{23}+F_{45}\; , \; Y_1=(0,i)\to -F_{23}+F_{45}$ and similarly for
$X_i,Y_i,i=2,3$. Thus, after renumbering the coordinates, it follows
that $\Kmo=(e^{-3it},e^{it})$ and $\Kpo=(e^{jt},e^{-3jt})$. For the
group picture to be consistent, we are left with:
$$
   \Delta Q \subset\{ (e^{-3it},e^{it})\cdot H\;
   ,\;
   (e^{jt},e^{-3jt})\cdot H
   \}\subset\S^3\times\S^3.
  $$

In order to determine the functions describing the metric, let
$\bar{X}_i,\bar{Y}_i$ be the action fields on $\SO(5)$ and
$X^*_i,Y^*_i$ those on $\SO(5)/\SO(3)$. To compute their length at
$\gamma(t)$, we translate them back to the identity with the
isometric left translation by $\exp(tF_{12})^{-1}$. We thus obtain:
$$\bar{X}_1(t)=\Ad(-\exp(tF_{12}))(F_{23}+F_{45})=-\sin(t)F_{13}+\cos(t)F_{23}+F_{45},$$
$$\bar{Y}_1(t)=\Ad(-\exp(tF_{12}))(-F_{23}+F_{45})=\sin(t)F_{13}-\cos(t)F_{23}+F_{45}.$$

Since $X^*_1=\bar{X}_1-\frac 1 5 \langle \bar{X}_1,H_1\rangle H_1 =
\bar{X}_1- \frac 1 5 (2\cos(t)+1)H_1$ and $Y^*_1=\bar{Y}_1- \frac 1
5 (-2\cos(t)+1)H_1$ we have:
$$f_1=\frac 1 5 (5+4\sin^2(t)-4\cos(t))\; , \; g_1=\frac 1 5 (5+4\sin^2(t)+4\cos(t))
\; , \; h_1=-\frac 1 5 (1-4\cos^2(t)).$$

The remaining functions are determined by Weyl group symmetry.
Similarly to the example of $\Sph^7$, we see that $w_-=e^{i\frac \pi
4}(1,1)$
 and  $w_+=e^{j\frac
\pi 4}(1,1)$ and hence $(w_+w_-)^3=(-1,-1)\in H$. Thus in this case
$W=D_3$. But conjugation by $w_+w_-$ behaves as before and hence all
functions satisfy the same Weyl symmetry as in \eqref{Weyl}.

\bigskip

It is now interesting to compare the metrics in these two examples,
which we do in a sequence of pictures. Figure 1 shows all 9
functions between two singular orbits, clearly not very instructive.
Figure 2 illustrates the effects of Weyl symmetry in these pictures
in a typical case of the 3 $g$ functions for the Berger space. Thus
Figure 3, which shows $f=f_1 ,\, g=g_1 ,\, h=h_1$ on $[0,3L]$,
encodes all the geometry of the space. As was discovered by K.Grove,
B.Wilking and the author, the  positivity of the sectional
curvatures $\sec(\gamma'(t),X^*)\; , X\in\fg$ implies that the
inverse of the metric tensor is a convex matrix. Figure 4 shows the
functions $F_1,G_1,H_1$ in the inverse of ${\tiny \left(
                 \begin{array}{cc}
                   f_1    & h_1 \\
                   h_1 & g_1 \\
                 \end{array}
               \right)  }$
               on the interval $[0,3L]$.
Smoothness conditions are now encoded in the behavior of the
functions as $t\to 0$ and $t\to 3L$.

\bigskip

\begin{center} $M=E^7_p$ \text{ with } $\bar{G}=\SO(3)\times\S^3$.
\end{center}

\bigskip

Next, we examine a family of biquotients among the 7-dimensional
Eschenburg spaces. Define
$$E_p :=\SU(3)/\!/\S^1_p=  \diag(z, z,
z^{p})\backslash \SU(3)/ \diag(1, 1, z^{p+2})^{-1},$$ where it is
understood that $\S^1=\{z\mid |z|=1\}$ acts on $\SU(3)$
simultaneously on the left and on the right. Up to equivalence, we
can assume that $p\ge 0$ and Eschenburg showed that it admits a
metric with \psc\ if $p\ge 1$. The positively curved metric is
obtained by scaling the biinvariant metric $\ml A,B\mr = -\frac 1 2
\Real \tr AB$ on $\SU(3)$ in direction of the subgroup
$\diag(A,\det\bar{A})\; , A\in \U(2)$ by an amount $\e<1$. The group
$G=\SU(2)\times\SU(2)$ acts on $E_p$ by multiplying on the left and
on the right in the first two  coordinates since it clearly commutes
with the circle action, and we claim this action is \coo. Indeed, we
can first divide by the second $\SU(2)$ and the action of the first
$\SU(2)$, since $(\{e\}\times\SU(2))\cdot\S^1_p=\U(2)$, is then an
action on $\SU(3)/\U(2)=\CP^2$ which fixes the identity coset
$e\cdot\U(2)$ and acts transitively on the normal sphere. Thus this
action on $\CP^2$, and hence also the action of $G$ on $E_p$, is
\coo.  Notice that if $p$ is even, the action becomes effectively an
action by $\SO(3)\times\SU(2)$, whereas if $p$ is odd, by
$\SU(2)\times\SO(3)$.

One singular point is clearly the image of the identity matrix,
$x_-=e\cdot \S^1_p$,  with $\Km=\{(g,\pm g)\mid g\in\SU(2)\}$. One
easily checks that in the modified metric on $\SU(3)$ the one
parameter group $\exp(tE_{13})$, being orthogonal to $\U(2)$, is
still a (unit speed) geodesic in $\SU(3)$ (see, e.g., \cite{DZ}).
Since it is also orthogonal to the orbit of $G$ at $x_-$, its
projection $\gamma(t)$ into $\SU(3)/\!/\S^1_p$ is a geodesic
orthogonal to all orbits. Its projection to $\CP^2$ is also a normal
geodesic and the induced $\SU(2)$ action has isotropy $\SU(2)$ at
$t=0$ and isotropy $\S^1$ at $t=\pi/2$ and the principal isotropy is
trivial. Thus the same holds for the action of $\bar{G}$ on $E_p$,
in particular $L=\pi/2$. For the explicit embeddings, one easily
checks (\cite{Zi1},\cite{GSZ}) that the isotropy group of $G$ at
$x_+=\gamma(\pi/2)$ is equal to $\Kp=(e^{i(p+1)t},e^{ipt})$.  Hence
we obtain the group diagram:
$$
   \Z_2=((-1)^{p+1},(-1)^{p})\subset \{ \Delta\S^3\cdot H \;
   ,\; (e^{i(p+1)t},e^{ipt})\}\subset\S^3\times\S^3.
  $$

  To compute the functions describing the metric, we identify
  $\S^3\times\S^3$ as before with $\SU(2)\times\SU(2)$ and translate
  the action fields at $\gamma(t)$ back to the identity. We then
  obtain:
  \begin{align*}
  \bar{X}^*_1&=\Ad(\exp(-tE_{13}))\diag(i,-i,0)=\diag(i\cos^2(t),-i,i\sin^2(t))
  +\cos(t)\sin(t)I_{13},\\
  \bar{X}^*_2&=\Ad(\exp(-tE_{13}))E_{12}=\cos(t)E_{12}-\sin(t)E_{23},\\
  \bar{X}^*_3&=\Ad(\exp(-tE_{13}))I_{12}=\cos(t)I_{12}+\sin(t)I_{23},\\
  \bar{Y}^*_1&=-\diag(i,-i,0)\; , \; \bar{Y}^*_2=-E_{12}\; , \;
  \bar{Y}^*_3=-I_{12},
  \end{align*}
  where we used the notation $I_{kl}$ for a matrix in $\fsu(3)$ with
  $i$ in entry $kl$ and $lk$ and $0$ elsewhere. The
  vertical space of the Riemannian submersion
  $\pi\colon\SU(3)\to\SU(3)/\!/\S^1_p$ (translated back to the identity) is spanned by
  \begin{align*}v&=\Ad(\exp(-tE_{13}))i\diag(1,1,p)-i\diag(0,0,p+2)\\
  &=
  i\diag(\cos^2(t)+p\sin^2(t),1, \sin^2(t)+p\cos^2(t)-p-2)
  +(1-p)\cos(t)\sin(t)I_{13},
  \end{align*}
  whose length in the Eschenburg metric is
  $$|v|^2=3\e + (1-p)^2(1-\e)\cos^2(t)\sin^2(t)+\e(p+2)(p-1)\sin^2(t).$$

  Notice that $\bar{X}_2,\bar{X}_3,\bar{Y}_2,\bar{Y}_3$ are already
  horizontal with respect to the Riemannian submersion $\pi$ but that we need to subtract the vertical component
  from $\bar{X}_1$ and $\bar{Y}_1$. A computation now shows that:
\begin{align*} f_1&= \frac \epsilon {4\alpha}\left[
\left(3\epsilon(p-2)^2-4p^2+8p-16\right) \cos^4 (t)+\right.\\
&\hspace{180pt}+ \left.\left(4p^2-8p+16-6\epsilon p(p-2)\,\right)
\cos^2(t)+3\epsilon
p^2\right],\\
g_1&= \frac \epsilon {4\alpha}\left[ (p-1)^2(3\epsilon-4) \cos^4
(t)+ (2p-2)(2p-2-3\epsilon(p+1)\, ) \cos^2(t)+3\epsilon
(p+1)^2\right],\\
h_1&=- \frac \epsilon {4\alpha}\left[(p-1)(3\epsilon(p-2)-4p+4)
\cos^4 (t)+ \right.\\
&\hspace{180pt} + \left.(4(p-1)^2-6\epsilon(p^2-p-1)\,)
\cos^2(t)+3\epsilon
p(p+1)\right],\\
&\hspace{40pt}f_2=f_3= 1 + (\epsilon - 1)\cos^2(t)  \; , \; g_2=g_3=
\epsilon \; , \; h_2=h_3= -\epsilon \cos(t).
\end{align*}

\no where $\alpha=(p-1)^2(\epsilon-1) \cos^4(t)+
(p-1)(p-1-\epsilon(2p+1)\,) \cos^2(t) +\epsilon (p^2+p+1)$ and all
other inner products are $0$.

\smallskip

Notice that $p=1$ is a  special case since the Eschenburg space is
now simply the homogeneous Aloff-Wallach space
$W^7=\SU(3)/\diag(z,z,\bar{z}^2)$ and the functions  are given by:
\begin{align*}f_1&=\frac 1 4 \left[(\epsilon-4)
\cos^4(t)+2(\epsilon+2)\cos^2(t)+\epsilon \right]\; , \; f_2=f_3= 1
+ (\epsilon - 1)\cos^2(t),\\
 g_1&= g_2=g_3= \epsilon \; , \;h_1=-\frac \epsilon 2 (\cos^2(t)+1)
\; , \; h_2=h_3= -\epsilon \cos(t).
\end{align*}

A major difference with the previous two cases  lies in the Weyl
group. Clearly $w_-=(-1,-1)$ and $w_+=(i^{p+1},i^p)$. Thus
$(w_+w_-)^2\in H$ and hence $W=D_2$. The Weyl group elements
multiply each of the natural basis vectors in $T_{p_\pm}B_\pm$ with
$\pm 1$ and hence Weyl group symmetry simply says that all
(non-collapsing) functions must be even at $t=0$ and at $t=L$, in
particular their first derivatives must vanish. Thus any of the
relationships between different functions that was so useful in the
previous cases is lost. Also, notice that, unlike in the previous
two examples, the vanishing of the remaining inner products is not
forced  anymore by the action of the isotropy group since it acts
trivially on $G/H$. The basic behavior of the functions is
illustrated in Figure 5  for the Aloff-Wallach space and the
Eschenburg space with a typical value of $p=10$ and $\e=\frac 1 2$.
Here we have drawn the graphs on $[0,4L]$, i.e., once around the
closed geodesic, for better comparison. As $p\to\infty$, the
functions converge, but the limiting metric is not smooth at the
singular orbits.

\bigskip

\begin{center} $W^7_{(2)}$ \text{ with } $\bar{G}=\SO(3)\SO(3)$.
\end{center}

\bigskip

The Aloff-Wallach space $W^7=\SU(3)/\diag(z,z,\bar{z}^2)$ has a
second cohomogeneity one action by combining right multiplication by
$\SU(2)$ as in the previous example with left multiplication by
$\SO(3)\subset\SU(3)$. Observe that the right action by $\SU(2)$ is
effectively an action of $\SO(3)=\U(2)/Z(\U(2))$ and that the action
is free with quotient $\SU(3)/\U(2)=\CP^2$. The left action by
$\SO(3)$ then induces an action on the quotient  which has to be the
\coa\ by $\SO(3)$ on $\CP^2$ mentioned earlier, since there exists
only one $\SO(3)$ in $\SU(3)$. In particular, the
$\bar{G}=\SO(3)\SO(3)$ action on $W^7$ is \coo, which also
determines the isomorphism type of the isotropy
 groups:
$\Km=\O(2),\; \Kp=\SO(2)$ and $H=\Z_2$.

 We can
choose $x_-$ again to be the identity coset since the isotropy is
$\SO(3)\cap\SU(2)\cdot\diag(z,z,\bar{z}^2)= \SO(3)\cap\U(2)=\O(2)$.
 The tangent space to $B_-$ is spanned by
$E_{ij}, I_{12}$ and $\diag(i,-i,0)$ and hence
$\gamma(t)=\exp(tI_{13})\cdot\diag(z,z,\bar{z}^2)$ is a unit speed
geodesic in $\SU(3)/\diag(z,z,\bar{z}^2)$ orthogonal to all orbits.
Since the projection to $\CP^2$ is a normal geodesic orthogonal to
the orbits of the \coa\  on $\CP^2$, it follows that the singular
isotropy groups occur at $t=0$ and $t=\pi/4$ and thus $L=\pi/4$.
Instead of trying to compute the embedding of these isotropy groups
directly, we will use the functions describing the metric instead.

We first compute the action fields $X^*_i$. For comparison, we again
consider the action as an (ineffective) action by $\S^3\times\S^3$,
and since the two fold cover $\SU(2)\to\SO(3)$ multiplies the
natural basis vectors by $2$, we find:
\begin{align*}\bar{X}_1(t)&=\Ad(-\exp(tI_{13}))(2E_{12})=-2\cos(t)E_{12}-2\sin(t)I_{23},\\
\bar{X}_2(t)&=\Ad(-\exp(tI_{13}))(2E_{13})=-2\cos(2t)E_{13}+2\sin(2t)\diag(i,0,-i),\\
\bar{X}_3(t)&=\Ad(-\exp(tI_{13}))
(2E_{23})=2\sin(t)I_{12}-2\cos(t)E_{23}.
\end{align*}

\no Notice that the component of $\diag(i,0,-i)$ orthogonal to
$\diag(z,z,\bar{z}^2)$  is $\frac 1 2 \diag(i,-i,0)$ and that
$\bar{X}_1(t)$ and $\bar{X}_3(t)$ are orthogonal already.

For the right action fields we have $\bar{Y}^*_1=-E_{12} \; , \;
\bar{Y}^*_2= -\diag(i,-i,0)\; , \;
  \bar{Y}^*_3=-I_{12}$ and thus:
\begin{align*}
 f_1&= 4 \sin^2(t)  + 4 \epsilon \cos^2(t) \; , \;g_1=\epsilon
\; , \;  h_1= 2 \epsilon \cos(t),\\
 f_2&=4 \cos^2(2t)  +  \epsilon
\sin^2(2t)  \; , \;g_2=\epsilon\; , \;h_2= -\epsilon \sin(2 t), \\
f_3&=4 \cos^2(t) + 4 \epsilon \sin^2(t)  \; , \; g_3=\epsilon\; ,
\;h_3=-2 \epsilon \sin(t),
\end{align*}
with all other inner products being $0$. Since $X_1^*(0)-2Y_1^*(0)=
0$ and $X^*_2(\pi/4)+Y^*_2(\pi/4)= 0$ it follows that
$\Kmo=(e^{it},e^{-2it})\cdot H$ and $\Kpo=(e^{jt},e^{jt})\cdot H$.
Since the isomorphism type of the isotropy groups of the action is
already determined, this leaves only the following possibility for
its group diagram:
$$
   \Z_4\oplus\Z_2=\{(\pm
1, \pm 1) , (\pm j , \pm j)\} \subset\{ (e^{it},e^{-2it})\cdot H\;
   ,\;
   (e^{jt},e^{jt})\cdot H
   \}\subset\S^3\times\S^3.
  $$
  For the Weyl group we have $w_-=(i,-1)   $,
   $w_+=(e^{j\frac \pi 4},e^{j\frac \pi 4})$ and  $(w_+w_-)^2=(-1,j)$
    and hence $W=D_4$. The functions $f_i$ and $g_i$ satisfy the
    same Weyl symmetry as in \eqref{Weyl2}, but for $h_i$ we have:
$$ h_3(t)=-h_1(-t+2L)=h_1(t+2L)\; , \; h_2(t)=-h_2(-t)=h_2(-t+2L) .$$
It is interesting to observe this modified Weyl symmetry behavior in
Figure 6 (for a typical value of $\e=\frac 1 2$).

\smallskip

The above metric with $\e=1$ is the one induced by the biinvariant
metric, which has non-negative curvature, but not  positive. For
$\e=2$, we obtain a 3-Sasakian metric (see Section 4) after dividing
the metric by $2$, i.e., multiplying all functions by $\tfrac 1 2$
and  replacing the parameter $t$ by $\sqrt{2}\; t$. The second
$\SU(2)$ factor is then the 3-Sasakian action.

\smallskip

    For later purposes we note that, up to conjugation, this group
    diagram can also be written as:
$$
   \Z_4\oplus\Z_2=\{(\pm
1, \pm 1) , (\pm i , \pm i)\} \subset\{ (e^{it},e^{it})\cdot H\;
   ,\;
   (e^{jt},e^{2jt})\cdot H
   \}\subset\S^3\times\S^3.
  $$

\smallskip

In Figure 7 we  show the graphs for the functions in the inverse
matrix on $W^7_{(1)}$ and $W^7_{(2)}$. We only include the most
interesting case of the one's with index 1, although these do not
determine the remaining ones as was the case for the sphere and the
Berger space.

\bigskip

\section{Classification of cohomogeneity one manifolds with positive
curvature}

\bigskip

In this section we describe the classification result in \cite{GWZ}.
In even dimensions, \pcu\  \com s were classified by L.Verdiani
\cite{V1,V2}. Here only rank one symmetric spaces occur. The actions
for these spaces though are numerous and have been classified in
\cite{HL,Iw1,Iw2,Uc}. In odd dimensions, \coa s on spheres are even
more numerous. The classification of course must also contain all of
the examples described in the previous section, as well as the \coa\
by $\SU(4)$ on the Bazaikin spaces $B^{13}_p= \diag (z,z,z,z,z^p)
\backslash \SU(5)/ \diag (z^{{p+4}},A)^{-1}$, where $A\in
\Sp(2)\subset \SU(4)\subset\SU(5)$. Here $\SU(4)$ acts by
multiplication on the left in the first $4$ coordinates.
 Encouragingly, a
series of ``candidates`` $P_k, Q_k$ and $R$ emerges in dimension $7$
for which it is not yet known whether they can carry a \coo\ metric
with \pc. They are \coo\ under an action of $\S^3\times\S^3$, and
are defined as \coms\ in terms of their isotropy groups. For $P_k$
the group diagram is
$$
   \Delta Q \subset\{ (e^{it},e^{it})\cdot H\;
   ,\;
   (e^{j(1+2k)t},e^{j(1-2k)t})\cdot H
   \}\subset\S^3\times\S^3,
  $$
  and for $Q_k$ it is
$$
   \{(\pm
1, \pm 1) , (\pm i , \pm i)\} \subset\{ (e^{it},e^{it})\cdot H\;
   ,\;
   (e^{jkt},e^{j(k+1)t})\cdot H
   \}\subset\S^3\times\S^3,
  $$
  whereas for $R$ we have
  $$
   \{(\pm
1, \pm 1) , (\pm i , \pm i)\} \subset\{ (e^{it},e^{2it})\cdot H\;
   ,\;
   (e^{3jt},e^{jt})\cdot H
   \}\subset\S^3\times\S^3.
  $$
  Notice that the action of $\S^3\times\S^3$ on $P_k$ is effectively
  an action by $\SO(4)$, and the one on $Q_k$ and $R$ by
  $\SO(3)\times\SO(3)$.

\begin{thm}[L.Verdiani,\; K.Grove-B.Wilking-W.Ziller]\label{classify}
A  simply connected compact cohomogeneity one manifold with an
invariant metric of positive sectional curvature is equivariantly
diffeomorphic to one of the following:
\begin{itemize}
\item
A compact rank one symmetric space with an isometric action,
\item
     One of $E^7_p,  B^{13}_p$ or
$B^7$,
\item
One of the  $7$-manifolds $P_k, Q_k$, or $R$,
\end{itemize}
\no with one of the actions described above.
\end{thm}

The first in each sequence $P_k,Q_k$ admit an invariant metric with
positive curvature since from the group diagrams in Section 2 it
follows that $P_1=\Sph^7$ and $Q_1=W^7_{(2)}$. By the main result in
\cite{GZ}, the manifolds $P_k, Q_k, R$ all carry an invariant metric
with \nnc\ since the \coa s have singular orbits of codimensions 2.
Recall that the \coo\ action on $B^7$ looks like those for $P_k$
with slopes $(-3,1)$ and
 $(1,-3)$. In some tantalizing sense then, the exceptional Berger
manifold $B^7$ is associated with the $P_k$ family in an analogues
way as the exceptional candidate $R$ is associated with the $Q_k$
family.

\vspace{10pt}

The candidates also  have interesting topological properties. In
\cite{GWZ} it was shown that the manifolds $P_k$ are two-connected
with $\pi_3(P_k) =\Z_k$ and that $Q_k$ has the same cohomology ring
as $E_k$.
  The fact that the manifolds $P_k$ are 2-connected is
particularly significant since by the finiteness theorem of
Petrunin-Tuschmann \cite{PT} and Fang-Rong \cite{FR},  there exist
only finitely many diffeomorphism types of  2-connected \pcu\
manifolds, if one specifies the dimension and the pinching constant,
i.e. $\delta$ with $\delta \le sec \le 1$. Thus, if $P_k$ admit
positive curvature metrics, the pinching constants $\delta_k$
necessarily go to 0 as $k\to \infty$, and $P_k$ would be the first
examples of this type.

\bigskip

It is remarkable that all non-linear actions in \tref{classify},
apart from the Bazaikin spaces $B_p^{13}$, occur in dimension $7$
and are \co\ one under a group locally isomorphic to
$\S^3\times\S^3$. It is also remarkable that in \pc\ only the above
slopes are allowed, whereas for arbitrary slopes one has an
invariant metric with \nnc\ by \cite{GZ}. We will
 give a proof of \tref{classify} in this special case
of $G=\S^3\times\S^3$ since this case is clearly of particular
interest.

\smallskip

Three important ingredients in the proof of the classification
\tref{classify} are given by:

\begin{itemize}
\item The normal geodesic is closed, or equivalently, the Weyl group
is finite.
\item The action is linearly primitive, i.e. the Lie algebras of all
singular
 isotropy groups  along a fixed normal geodesic generate $\fg$ as
vector spaces.
\item The action is group primitive, i.e. the groups $\Km$ and $\Kp$
generate $G$ as subgroups and so do $\Km$ and $n\Kp n^{-1}$ for any
$n\in N(H)_{\subo}$.
\end{itemize}

\smallskip

For $G=\S^3\times\S^3$, the classification is based on the following
Lemma.

\begin{lem}\label{slopes}
Let $G=\S^3\times\S^3$ act by \coo\ on the \pcu\ manifold $M$. If
$G/K$ is a singular orbit, $G/\Ko=\S^3\times\S^3/(e^{ipt},e^{iqt})$
with $p,q\ne 0, (p,q)=1$, and $H\cap \Ko=\Z_k$, we have:
\begin{itemize}
\item[(a)]  $k\ge 2$ and if $k=2$, then $|p+q|=1$ or $|p-q|=1$,
\item[(b)] If $k>2$, then  $(p,q)=(\pm 1,\pm 1)$ or $|2p+2q|=k$ or $|2p-2q|=k$.
\end{itemize}
If furthermore $G/\Kmo=\S^3\times\S^3/(e^{ip_-t},e^{iq_-t})$,
$G/\Kpo=\S^3\times\S^3/(e^{jp_+t},e^{jq_+t})$ and, up to conjugacy,
$H=\Delta Q$ or $H= \{(\pm 1, \pm 1) , (\pm i , \pm i)\}$, then
$\min\{|p_+|,|p_-|\}=\min\{|q_+|,|q_-|\}=1$.
\end{lem}

\begin{proof}
The main ingredient in the proof of (a) and (b) is the equivariance
of the second fundamental form of $G/K$ regarded as a $K$
equivariant linear map $B\colon S^2T\to T^\perp$. The non-trivial
irreducible representations
       of $\S^1=\{e^{i\gt } \mid \gt \in \R \}$ consist of two dimensional
      representations given by multiplication by $e^{in\gt}$ on $\C$,
      called a weight $n$ representation. The action of $\Ko$ on
$T^\perp=\R^2$
      will have weight $k$ if $H\cap \Ko=\Z_k$, since $\Z_k$ is necessarily  the
      ineffective kernel.
The action of $\Ko$ on $T$ on the other hand has weights $0$ on
$W_{\subo}=\spam\{(-qi,pi)\}$, weight $2p$ on the two plane
$W_1=\spam\{(j,0),(k,0)\}$ and weight $2q$ on
$W_2=\spam\{(0,j),(0,k)\}$. The action on $S^2(W_1\oplus W_2)$ has
therefore weights $0$ and $4p$ on $S^2W_1$, $0$ and $4q$ on $S^2W_2$
and $2p+2q$ and $2p-2q$ on $W_1\otimes W_2$.

Let us first assume that $(p,q)\neq (\pm 1,\pm 1)$.  Then for any
homogeneous metric on $G/\Ko$ there exists a vector $w_1\in W_1$ and
$w_2\in W_2$ such that the 2-plane spanned by $w_1$ and $w_2$
tangent to $G/K$ has curvature 0 intrinsically. Indeed, since $p\neq
\pm q$, $\Ad(\Ko)$ invariance of the metric on $G/\Ko$ implies that
the two planes span$\{(j,0) , (0,j)\}$ and span$\{(k,0) , (0,k)\}$
and the line $W_{\subo}$  are orthogonal to each other. Hence
$\Ad((j,j))$ induces an isometry on $G/\Ko$,
 which implies that
 the two plane spanned by
  $w_1=(j,0)\in W_1$ and $w_2=(0,j)\in W_2$
  is the tangent space of
 the fixed point set of $\Ad((j,j))$ and thus
has curvature 0 since the fixed point set in $G/\Ko$ is a 2-torus
with a left invariant metric. Since $(p,q)\neq (\pm 1,\pm 1)$
 at least one of the numbers
$4p$ or $ 4q$ is not equal to the normal weight $k>0$. The
 equivariance of the second fundamental
form then implies that  $B_{S^2W_i}$ vanishes for at least one $i$
and hence by the Gauss equations  $B(w_1,w_2)\neq 0$ for the above
vectors $w_1$ and $w_2$. Using the equivariance of the second
fundamental form once more we see that $W_1\otimes W_2$ contains a
subrepresentation whose weight is equal to the normal weight $k$.
Hence, $|2p+2q|=k$ or $|2p-2q|=k$. In particular, $k\ge 2$.

It remains to show that if $(p,q)=(\pm 1,\pm 1)$, then $k>2$. We
first show that we still have a 2-plane as above with 0-curvature.
Indeed, $\Ad(\Ko)$ invariance implies that the inner products
between $W_1$ and $W_2$ are given by $\langle (X,0),(0,Y)\rangle =
\langle \phi(X),Y\rangle$  where $\phi\colon W_1\to W_2$ is an
$\Ad(\Ko)$ equivariant map. Hence, if we choose $j'=\phi(j)$ and
$k'=\phi(k)$, the two planes span$\{(j,0) , (0,j')\}$ and
span$\{(k,0) , (0,k')\}$ are orthogonal to each other, so that by
the same argument $w_1=(j,0)\in W_1$ and $w_2=(0,j')\in W_2$ span a
2-plane with curvature 0. We thus obtain again that $B_{S^2W_i}=0$,
hence $B(w_1,w_2)\neq 0$, which gives a contradiction if $k<4$.

\smallskip

To prove part (c), we use the following general fact about  an
isometric $G$ action on $M$.
 The strata, i.e., components in $M/G$ of orbits of the
same type $(K)$, are (locally) totally geodesic (cf.
\cite{grove:survey}). Indeed, by the slice theorem, such a component
 near the image of $p\in M$ with $G_p=K$ is  given by
the fixed point set $D^{K}\subset D/K\subset M/G $ where $D$ is a
slice at $p$. In the case of the $\S^3\times 1$ action on $M$, the
isotropy groups are given by the intersections of $\S^3\times 1$
with $\Kpm$ and $H$ since $\S^3\times 1$ is normal in $G$. Using the
special form of the principal isotropy group $H$, it follows that on
the regular part the isotropy groups are effectively trivial. On the
other hand, if $\min\{|q_+|,|q_-|\}>1$,  they are non-trivial along
$B_\pm$. This implies that the image of both $B_\pm$ in
$M/\S^3\times 1$ are totally geodesic. Since these strata are two
dimensional and $M/\S^3$ is four dimensional, both strata cannot be
totally geodesic according to Petrunin's analogue \cite{Pe} of
Frankel's theorem for Alexandrov spaces. This finishes our claim.
\end{proof}

\smallskip

We now use the classification of 7-dimensional compact simply
connected primitive \coms\ in \cite{Ho}. Although the use of this
classification is not necessary, it  simplifies the argument and
brings out its main points. Surprisingly there are, in addition to
numerous linear actions on $\Sph^7$,  only 6 (group) primitive
families in the classification, 5 of them with $G=\S^3\times \S^3$
and we apply \lref{slopes} to exclude from them the ones that do not
admit an invariant metric with \pc.

\bigskip

{\it Example 1. } The simplest primitive group diagram on a
7-manifold with $G=\S^3\times \S^3$ is given by
$$\{e\}\subset \{\Delta \S^3,(e^{ip\theta},e^{iq\theta})\}\subset \S^3\times
\S^3,$$ for any $p,q$. A modification of this example is
$$\Z_2=\{(1,-1)\}\subset \{\Delta \S^3\cdot H,(e^{ip\theta},e^{iq\theta})\}
\subset \S^3\times \S^3,$$ with  $p$ even and $q$ odd. Notice that
the second family is a two fold branched cover of the first.

In the first case the normal weight is $k=1$ and thus \lref{slopes}
(a) implies that there are no positively curved invariant metrics.
In the second case $k=2$ and it follows that $|p-q|=1$ or $|p+q|=1$,
which, up to automorphisms of $\S^3\times \S^3$, is the Eschenburg
space $E_p$. Notice though that in both cases we also need to
exclude the possibility that one of $p$ or $q$ is $0$, which is not
covered by \lref{slopes}. For this special case, one uses the
product Lemma \cite[Lemma 2.6]{GWZ}, which we will not discuss here.

\smallskip

{\it Example 2. } The second family has $H=\Z_4$:
$$ \langle (i,i) \rangle\subset \{(e^{ip_-\theta},e^{iq_-\theta})\cdot H ,
(e^{jp_+\theta},e^{jq_+\theta})\cdot H \}\subset \S^3\times \S^3,$$
where $p_-,q_-\equiv 1 \mod 4$, and $p_+,q_+$ are arbitrary. This
family is excluded altogether. Indeed, if $p_+$ and $q_+$ are both
odd, the normal weight at $B_+$ is $k=2$, which implies that
$|p_+\pm q_+|=1$, which is clearly impossible. If one is even, the
other odd, $k=1$ which is excluded by \lref{slopes}. If $p_+=0$ or
$q_+=0$, the action is not group primitive.

\smallskip

{\it Example 3. } The third family has $H=\Z_2\oplus\Z_4$:
$$ \{(\pm
1, \pm 1) , (\pm i , \pm i)\}\subset
\{(e^{ip_-\theta},e^{iq_-\theta})\cdot H ,
(e^{jp_+\theta},e^{jq_+\theta})\cdot H \}\subset \S^3\times \S^3,$$
where $p_-,q_-$ is odd, $p_+$  even and $q_+$  odd.  On the left,
$B_-$ has normal weight $k=4$ and thus $(p_-,q_-)=(\pm 1,\pm 1)$ or
$|p_-+q_-|=2$ or $|p_--q_-|=2$. On the right, $B_+$ has normal
weight $k=2$ and thus $|p_++q_+|=1$ or $|p_+-q_+|=1$. Notice also
that we can assume that all integers are positive by conjugating all
groups by $i$ or $j$ in one of the components, and observing that
$p_+=0$ is not group primitive. Together with \lref{slopes} (c),
this leaves only the possibilities $\{ (p_-,q_-)\, , \, (p_+,q_+)\}
= \{(1,3) , (2,1)\} $ or $\{ (p_-,q_-)\, , \, (p_+,q_+)\} = \{(1,1)
, (p_+,q_+)\} $ with $|q_+-p_+ |=  1$. The first case is the
exceptional manifold $R$. In the second case we can also assume that
$q_+
>p_+$ by interchanging the two $\S^3$ factors if necessary, and hence
$(p_+,q_+)=(p,p+1)$ with $p>0$.  This gives us the family $Q_k$.

\smallskip

{\it Example 4. } The last family has $H=\Delta Q$:
$$ \Delta Q\subset \{(e^{ip_-\theta},e^{iq_-\theta})\cdot H ,
(e^{jp_+\theta},e^{jq_+\theta})\cdot H \}\subset \S^3\times \S^3,$$
where $p_\pm,q_\pm\equiv 1 \mod 4$. Now the weights on both normal
spaces are $4$ and hence $|q_\pm + p_\pm|= 2$  or
$(p_\pm,q_\pm)=(1,1)$.
 Combining with \lref{slopes} (c) yields only two
possibilities. Either $\{ (p_-,q_-)\, , \, (p_+,q_+)\} = \{(1,-3) ,
(-3,1)\} $ or $\{(1,1) , (p_+,q_+)\}$ with  $q_+ +p_+=2$. The first
case is the Berger space $B^7$ and in the second case we can arrange
that $\{ (p_-,q_-)\, , \, (p_+,q_+)\} = \{(1,1) , (1+2k,1-2k)\}$
with $k\ge 0$. The case $k=0$ is excluded  since it would not be
group primitive. Thus we obtain the family $P_k$.

\smallskip

{\it Remark. } There is only one further family of compact simply
connected 7-dimensional primitive \com s, given by the action of
$\S^1\times \S^3\times \S^3$ on the Kervaire sphere. For this action
it was shown in \cite{BH} that it cannot admit an invariant metric
with \pc\ unless it is a linear action on a sphere. In \cite{GVWZ}
it was shown that in most cases it does not even admit an invariant
metric with \nnc. If we also allow non-primitive 7-dimensional \com
s, one finds 9 further families in \cite{Ho}. He also shows that the
only \com\ in dimension 7 or below (primitive or not) where it is
not yet known if it admits an invariant metric with \nnc, are the
two families in Example 1.

We also mention that in dimension 7, one finds a classification of
\pcu\ \com\ in \cite{PV1,PV2} in the case where the group is not
locally isomorphic to $\S^3\times \S^3$ and that a classification in
dimensions 6 and below was obtained in \cite{Se}.

\smallskip

\section{Candidates and Hitchin metrics}

\bigskip

The two families of \com s $P_k$ and $Q_k$ have another remarkable
and unexpected property. They carry a natural metric on them which
is 3-Sasakian and can be regarded as an orbifold principal bundle
over $\Sph^4$ or $\CP^2$, equipped with a self dual Einstein metric.

\smallskip

Before we discuss these metric, we give a description of our
candidates using the language of self duality. If we consider an
oriented four manifold $M^4$ equipped with a metric, we can use the
Hodge star operator $\star\colon \Lambda^2T^*\to \Lambda^2T^*$
 to define self dual and anti self dual 2-forms, i.e., forms $\omega$ with
  $\star\omega=\pm
 \omega$. They each form a 3-dimensional vector bundle $\Lambda_\pm^2T^*(M)$
  over $M$  with a fiber metric induced by the metric on $M$. Thus
 their principal frame bundles are two natural $\SO(3)$ principal bundles
 associated to the tangent bundle of $M$.
 The star operator depends
  on the conformal class of the metric but the isomorphism type of the vector bundles only
  depend on the given orientation. The same construction can be
  carried out for oriented orbifolds equipped with an orbifold metric
  since such a metric is locally the quotient of a smooth Riemannian
  metric under a finite group of isometries. The principal bundles
  are then orbifold bundles. In particular $\SO(3)$ will in general
  only act almost freely (i.e., all isotropy groups are finite) and
  the total space may only be an orbifold.

  We use this construction now for the following orbifold structure
  $ O_k$ on $\Sph^4$. Consider the \coa\ of $\SO(3)$ on $\Sph^4$
  described in Section 2. The two singular orbits are Veronese
  embeddings of $\RP^2$ into $\Sph^4$. We define an orbifold $O_k$ by
  requiring that it be smooth along the regular orbits and one
  singular orbit, but along the second singular orbit it is smooth in the
  orbit direction and has an angle $2\pi/k$ normal to it. Since the
  normal space is two dimensional, this orbifold is still
  homeomorphic to $\Sph^4$. When $k$ is even, we have the $\SO(3)$ equivariant
   two fold
  branched cover  $\CP^2\to\Sph^4$ mentioned in Section 2, and $O_{2k}$ pulls back to
  an orbifold structure on $\CP^2$ with angle $2\pi/k$ normal
  to the real points $\RP^2\subset\CP^2$. If we equip the orbifold
  with an orientation and with an orbifold metric invariant under the $\SO(3)$
  action, we can define the $\SO(3)$ principal bundle
  $H_k$
  of the vector bundle of self dual  2-forms on
  $O_k$. The orientation we choose is adapted to the \coa\ as
  follows. Recall that we have a natural basis in the orbit
  direction corresponding to the action fields $X_1^*=E_{12}^*\; ,
  X_2^*=E_{13}^*\; , X_3^*=E_{23}^*\; $  and
  we choose the orientation defined by $\gamma',X_1^*,X_2^*,X_3^*$ where
  $\gamma$ is the
   normal geodesic chosen in Section 2. Along the
  singular orbits $X_1^*$ respectively $X_2^*$ vanishes and should be replaced by
   the derivative of the Jacobi field induced by their action
   fields. The isometric \coa\ of $\SO(3)$ on $O_k$ clearly lifts to an action
   on the bundle of self dual 2-forms and thus onto the principal bundle
    $H_k$. It commutes with the principal bundle action of
   $\SO(3)$ and together they form an $\SO(3)\times\SO(3)$ \coa\ on
   $H_k$. We now show:

\begin{thm}\label{PQ-H}
The total space $H_k$ of the $\SO(3)$ principal orbifold bundle of
self dual 2-forms on $O_k$ is smooth and the \com s  $P_k$ and $Q_k$
are
     equivariantly diffeomorphic to  the (2-fold) universal
covers of  $H_{2k-1}$ and $H_{2k}$ respectively.
\end{thm}

\begin{proof} Recall that for the $\SO(3)$ \coo\ action on $\Sph^4$ in
Section 2, the isotropy groups are given by $H=\Z_2\oplus\Z_2
\subset\{\O(2) ,\O'(2)\}\subset\SO(3)$ where the two singular
isotropy groups are embedded in two different blocks. Since the
metric on $O_k$ is smooth near $B_-$, it follows that we still have
$\Km\cong \O(2)$, which we can assume is embedded in the upper
block, and hence $H\cong \Z_2\oplus\Z_2$ embedded as the set of
diagonal matrices in $\SO(3)$.  The normal angle along $B_+$ is
$2\pi/k$ and a neighborhood of $B_+$ can be described as follows.
The homomorphism $\phi_k\colon \SO(2)\to \SO(2)$, $\phi_k(A)= A^k$
gives rise to a homomorphism
 $j_k\colon \Kpo\simeq\SO(2) \to \SO(3)$, where $A\in \Kpo$ goes to
  $\phi_k(A)$
 followed by an embedding into the lower block of
 $\SO(3)$. We can now  define $\Kp=\O(2)$ for $k$
 odd, $\Kp=\O(2)\times\Z_2$
  for $k$ even and extend the homomorphism to $j_k\colon \Kp \to \SO(3)$
  such that  $\diag(1,-1)\in\O(2)$ goes to $\diag(-1,1,-1)$
  and the non-trivial element in $\Z_2$ goes to $\diag(1,-1,-1)$ when $k$ is even.
A neighborhood of the singular orbit on the right is then given by
     $D(B_+)=\SO(3)\times_{\Kp} D_+^2$
where $\Kp$ acts on $\SO(3)$ via $j_k$, $\Kpo$ acts on  $D_+^2$ via
$\phi_2$, $\diag(1,-1)\in\O(2)$ acts as a reflection, and $\Z_2$
acts trivially. Indeed, we then have $ \SO(3)\times_{\Kp} D_+^2=
\SO(3)\times_{(\Kp /\ker j_k)} (D_+^2/\ker j_k)$ with singular orbit
$\SO(3)/\O(2)$ and normal disk $D_+^2/\ker j_k=D_+^2/\Z_k$.
Furthermore, $\partial D(B_+)=\SO(3)\times_{\Kp} \Sph^1_+=\SO(3)/H$.


  The vector bundle of self dual two forms can  be
viewed as follows: Let P be the $\SO(4)$ principal bundle of
oriented orthonormal frames in the orbifold tangent bundle of $O_k$.
This frame bundle is a smooth manifold since the finite isometric
orbifold groups act freely on frames. $\SO(4)$ has two normal
subgroups $\SU(2)_-$ and $\SU(2)_+$, given by left and right
multiplication of unit quaternions, with
$\SO(4)/\SU(2)_\pm\simeq\SO(3)$. The $\SO(3)$ principal bundles
$P/\SU(2)_+$ and $P/\SU(2)_-$ are then the principal bundles for the
vector bundle of self dual and the vector bundle of anti self dual 2
forms. This is due to the fact that the splitting $\Lambda^2V\cong
\Lambda_-^2V\oplus\Lambda_+^2V$ for an oriented four dimensional
vector space corresponds to the splitting of Lie algebra ideals
$\fso(4)\cong \fso(3)\oplus\fso(3)$ under the isomorphism
$\Lambda^2V\cong\fso(4) $. Alternatively, we can first project under
the two fold cover $\SO(4)\to\SO(3)\SO(3)$ and then divide by one of
the $\SO(3)$ factors. The action of $\SO(4)$ on $P$ is only almost
free since $P/\SO(4)=O_k$ but we will show that both $\SU(2)_+$ and
$\SU(2)_-$ act freely on $P$, or equivalently, each $\SO(3)$ factor
in $\SO(3)\SO(3)$ acts freely on $P/\{-\Id\}$. This then implies
that $P/\SU(2)_+=H_k$ is indeed a smooth manifold.

\smallskip

The description of the disc bundle $D(B_+)$ gives rise to a
description of the corresponding  $\SO(4)$ frame bundle
$\SO(3)\times_{\Kp} \SO(4)$ where the action of $\Kp$ on $\SO(3)$ is
given by $j_k$ as above, and the action of $\Kpo$ on $\SO(4)$ is
given via $\SO(2)\subset\SO(4)\colon A\in\SO(2)\to
(\phi_k(A),\phi_2(A))$ acting on the splitting $T_+\oplus T^\perp_+$
into tangent space and normal space of the singular orbit. Similarly
for the left hand side where $k=1$.  On the left hand side  the
$X_1^*$ direction collapses, $T_-$ is oriented by $X_2^*,X_3^*$  and
$T^\perp_-$ by $\gamma'(0),X_1^*$.  On the right hand side the
$X_2^*$ direction collapses, $T_+$ is oriented by $X_3^*,X_1^*$ and
$T^\perp_+$ by $ \gamma'(L),X_2^*$. Furthermore,
$\SO(2)\subset\O(2)$ has  negative weights on $T_\pm$, where we have
endowed the isotropy groups on the left and on the right with
orientations induced by $X_1$ and $X_2$ respectively. Indeed,
$[E_{12},E_{13}]=-E_{23}$ on the left and $[E_{13},E_{23}]=-E_{12}$
on the right. On $T^{\perp}_-$, the weight is positive, and on
$T_+^\perp$ negative. Hence $\Kpmo\subset\SO(3)\SO(4)$ sits inside
the natural maximal torus in $\SO(3)\SO(4)$ with slopes $(1,-1,2)$
on the left, and $(k,-k,-2)$ on the right. To make this precise
metrically, we can choose as a metric on $H_k$, the natural
connection metric induced by the Levi Cevita connection on $O_k$. A
parallel frame is then a geodesic in this metric. By equivariance
under $H$, the unit vectors $X_i^*/|X_i^*|$ form such a parallel
frame.

 Under the homomorphism
$\SO(4)\to\SO(3)\SO(3)$ and the natural maximal tori in $\SO(4)$ and
in $\SO(3)\SO(3)$, a slope $(p,q)$ circle goes into one with slope
$(p+q,-p+q)$. Hence the slopes of $\Kpmo$ in $\SO(3)\SO(3)\SO(3)$
are $(1,1,3)$ on the left, and $(k,-(k+2),k-2)$ on the right. This
also implies that the second and third $\SO(3)$ factor each act
freely on $P/\{-\Id\}$. Here we have used the fact that we already
know that $\SO(4)$, and thus each $\SO(3)$, acts freely on the
regular part and hence freeness only needs to be checked in $\Kpo$.
Notice also that for $k$ even, all slopes in $\Kpo$ should be
divided by 2 to make the circle description effective. If we divide
by the third $\SO(3)$ to obtain $H_k$, the slopes are $(1,1)$ on the
left and $(k,-(k+2))$ on the right. Finally, notice that the
principle isotropy group of the $\SO(3)\SO(3)$ action on $H_k$ is
again $\Z_2\oplus\Z_2$ since this is true for the $\SO(3)$ action on
$O_k$ and $\SO(4)$ acts freely on the regular points in $P$. This
determines the group diagram. For $k=2m-1$, it is the group diagram
of the two fold subcover of $P_m$ obtained by dividing
$G=\S^3\times\S^3$ by its center. For $k=2m$, this is the group
picture of the two fold subcover of $Q_m$ obtained by adding a
component to all 3 isotropy groups (generated e.g. by $(j,j)$). This
finishes our proof.
\end{proof}

{\it Remarks. } (a) The proof also shows that the $\SO(3)$ principal
bundles $P/\SU(2)_-$ corresponding to the vector bundle of anti-self
dual two forms is smooth and has slopes $(1,3)$ on the left and
$(k,k-2)$ on the right.  Note that in the case of $k=3$ one obtains
the slopes for the exceptional manifold $B^7$ and in the case of
$k=4$ the ones for $R$ (up to 2-fold covers).

\smallskip

(b)  In the case of $k=2\ell$, we can regard $O_k$ as an orbifold
metric $O_{\ell}$ on $\CP^2$. In this case it follows that the
$\SO(3)$ principal bundle of the  bundle of self dual two forms is
$Q_\ell$ itself.

\bigskip

We now explain the relationship to the Hitchin metrics. Recall that
a metric on $M$ is called 3-Sasakian if $G=\SU(2)$ or $G=\SO(3)$
acts isometrically and almost freely with totally geodesic orbits of
curvature 1. Moreover, for $U$ tangent to the $\SU(2)$ orbits and
$X$ perpendicular,  $X\wedge U$ is required to be
 an eigenvector
of the curvature operator $\hat{R}$  with eigenvalue 1, in
particular  the sectional curvatures $\sec(X,U)$ are equal to 1. In
the case we are interested in, where the dimension of $M$ is $7$,
the quotient $B=M^7/G$   is 4-dimensional and its induced metric is
self-dual Einstein with positive scalar curvature, although it is in
general only an orbifold metric. Recall that a metric is called self
dual if the curvature operator satisfies $\hat{R}\circ
\star=\star\circ\hat{R}$.
 Conversely, given a self-dual Einstein
  orbifold metric on $B^4$ with positive scalar
curvature,  the $\SO(3)$ principal orbifold bundle of self dual
2-forms on $B^4$ has a 3-Sasakian orbifold metric  given by the
naturally defined Levi Cevita connection metric. See \cite{BG} for a
survey on this subject.

Recall that $\Sph^4$ and $\CP^2$, according to Hitchin,  are the
only smooth self dual Einstein 4-manifolds. The 3-Sasakian metrics
they give rise to are the metric on $\Sph^7(1)$ in the first case,
and in the second case  the metric on the Wallach space $W^7_{(2)}$
described in Section 2. However, in the more general context of
orbifolds, Hitchin constructed in \cite{Hi1} a sequence of self dual
Einstein orbifolds $O_k$ homeomorphic to $\Sph^4$, one for each
integer $k>0$. The metric is invariant under the \coa\ by $\SO(3)$
from Section 2  and has an orbifold singularity as in the orbifold
$O_k$ discussed earlier. The cases of $k=1,2$ correspond to the
smooth standard metrics
 on $\Sph^4$ and on
$\CP^2$ respectively. Hence the Hitchin metrics give rise to
3-Sasakian orbifold metrics on the seven dimensional orbifold
$H^7_k$. Here one needs to check that the orientation we chose above
agrees with the orientation in \cite{Hi1}. As we saw in \tref{PQ-H},
$H_k$ is actually smooth and the 3-Sasakian metric, as a quotient of
the smooth connection metric on the principal frame bundle, is also
smooth. Thus our candidates $P_k$ and $Q_k$ all admit a smooth
3-Sasakian metric. In the context of 3-Sasakian geometry, the
examples $P_k$ are particularly interesting since they are two
connected, and so far, the only known 2-connected example in
dimension 7, was $\Sph^7$.

\smallskip

It was shown by O.Dearricott in \cite{De} (see also \cite{CDR}) that
a 3-Sasakian metric, scaled down in direction of the principal
$G=\SO(3)$ or $\SU(2)$ orbits, has positive sectional curvature if
and only if the self dual Einstein orbifold base has positive
curvature. It is therefore interesting to examine the curvature
properties of the Hitchin metrics, which we will now discuss
shortly. The metric is described by the 3 functions
$T_i(t)=|X^*_i(\gamma(t))|^2$ along a normal geodesic $\gamma$,
since invariance under the isotropy group implies that these vectors
are orthogonal. It turns out that in order to solve the ODE along
$\gamma$ given by the condition that the metric is self dual
Einstein, it is convenient to change the arc length parameter from
$t$ to $r$. The metric is thus described by
$$g_{\gamma(r)}=f(r)dr^2+T_1(r)d\theta_1^2+T_2(r)d\theta_2^2+T_3(r)d\theta_3^2,
$$
where $d\theta_i$ is dual to $X_i$. In order to solve the ODE,
Hitchin uses complex algebraic geometry on the twistor space of
$O_k$. For general $k$, the solutions are explicit only in principal
and it is thus a tour de force to prove the required smoothness
properties of the metric. For small values of $k$ though, one finds
explicit solutions in \cite{Hi1} and \cite{Hi2}:

\bigskip

{ \it  Example 1. } The first non-smooth example is the Hitchin
metric with normal angle $2\pi/3$. Here the functions are algebraic:

\begin{align*}
T_1&=\frac{80r^2(r^6-2r^5-5r^4-15r^3-20r^2+13r+4)}{(3r^3+7r^2+r+1)^2
(3r^3-13r^2+r+1)},\\
 T_2&=\frac{5r(3r-1)(r-\beta )(r+\beta
+2)(2-3r-r^2+r^3+5\beta r)^2}{(3r^3+7r^2+r+1)^2(r^2+r-1)(r^2+r+4)}
                        ,\\
T_3&=  \frac{5r(3r-1)(r+\beta )(r-\beta +2)(2-3r-r^2+r^3-5\beta
r)^2}{(3r^3+7r^2+r+1)^2(r^2+r-1)(r^2+r+4)}                  ,\\
 f &=\frac{
5(3r-1)(r^2+r+4)(r+1)^2}{(r+r^2-1)/(3r^3+7r^2+r+1)^2},
\end{align*}
with  $\beta=\sqrt{\frac {r+r^2-1} r}$ and $\frac {\sqrt{5}-1} 2 \le
r \le 1$.

\bigskip

{ \it  Example 2. } The simplest example of a non-smooth Hitchin
metric has normal angle $2\pi/4$, where the functions are given by:

$$T_1=          \frac{(1-r^2)^2}{(1+r+r^2)(r+2)(2r+1)}    \; , \;
T_2=     \frac{1+r+r^2}{(r+2)(2r+1)^2}   \; , \;
T_3=\frac{r(1+r+r^2)}{(r+2)^2(2r+1)},$$
$$
f=\frac{1+r+r^2}{r(r+2)^2(2r+1)^2},$$ with $1\le r <\infty$.

\bigskip

{ \it  Example 3. } Finally, we have the Hitchin metric with normal
angle $2\pi/6$:

\begin{align*}T_1&=\frac{(3r^2+2r+1)(r^2+2r-1)^2(r^2-2r+3)(r^2+1) } {
(3r^2-2r+1)(r^2-2r-1)^2(r^2+2r+3)^2},\\
T_2&=\frac{ (3r^2-2r+1)(r^2-2r+3)(r+1)^3(r-1)}{
(3r^2+2r+1)(r^2+2r+3)^2(r^2-2r-1)},\\
 T_3&=\frac{
-4(3r^2-2r+1)(3r^2+2r+1)r}{
(r^2+2r+3)^2(r^2-2r-1)(r^2-2r+3)},\\
f&=\frac{(r+1)(r^2-2r+3)(3r^2+2r+1)(3r^2-2r+1) }{
r(1-r)(r^2-2r-1)^2(r^2+1)(r^2+2r+3)^2 },
\end{align*}
with $\sqrt{2}-1\le r \le 1$.

\bigskip

Although one can in principle use the methods in \cite{Hi1} to
determine the functions for larger values of $k$, they quickly
become even more complicated. The above 3 cases are sufficient
though to understand the behavior in general. If the functions are
given in arc length parameter, one has
$\sec(\gamma',X_i^*)=-f_i''/f_i$, where $f_i=\sqrt{T_i}$, and hence
positive curvature is equivalent to the concavity of $f_i$. In
Figures 8-10 we therefore have drawn a graph of the length functions
$f_i$ in arc length parameter, together with a graph of
$\sec(\gamma',X_i^*)$. The pictures are similar, but notice the
difference in scale. One sees that the non-smooth singular orbit
must occur at $t=L$ since it is necessarily totally geodesic, which
implies that the non-collapsing functions have 0 derivative. The
pictures  show that the function $f_1$, which vanishes at the smooth
singular orbit, is concave,
 whereas the other two are slightly convex near the smooth singular orbit.

\smallskip

For each of the 3 elements $g\in H$, the fixed point set of $g$ is a
2-sphere, since this is clearly true for the linear action on
$\Sph^4$ corresponding to $k=1$. They are isometric to each other
via an element of the Weyl group. Since the circle that commutes
with $g$ acts by isometries on the 2-sphere, it is rotationally
symmetric with an orbifold point at one of the poles. It has
positive curvature, except in a small region two thirds toward this
pole. Figure 11 shows the length of the action field induced by the
circle action on this 2-sphere (which is equal to $f_i/2$) in the
case of $k=3$ and $k=6$. Notice that it extends from $0$ to $3L$.
These 2-spheres can also be isometrically embedded as surfaces of
revolution in $3$-space, which we exhibit in Figure 12.

\smallskip

One can show that, as a consequence of being self dual Einstein,
$\sec(X^*_i,X^*_j)=\sec(\gamma',X_k^*)$, when $i,j,k$ are distinct,
and that if all 3 are positive, the  curvature of any 2-plane is
indeed positive also. Thus the Hitchin metric has positive curvature
wherever the above orbifold 2-sphere has positive curvature.  By
Dearricott's theorem, this implies that the induced 3-Sasakian
metric on our candidates, scaled down in direction of the principal
$\SO(3)$ orbits,  has \psc\ on half of the manifold.

Thus this metric does not yet give the desired metrics of positive
curvature on $P_k$ and $Q_k$. It is also tempting to think that, as
in the case of the known actions in Section 2, simple trigonometric
expressions for the 9 functions describing a metric  on $P_k$, $Q_k$
or $R$ might yield a metric with \pc\ on our candidates. But this
does not seem to be the case either, since already the smoothness
conditions and simple necessary convexity properties require
trigonometric functions that are quite complicated.

\smallskip

It is  intriguing that the (non-compact) space of 2-monopoles
studied by Atiyah and Hitchin in \cite{AH} has surprisingly similar
properties to the above metric. It carries a self dual Einstein
orbifold metric which in this case is Ricci flat, i.e., is
Hyperk\"ahler. It is invariant under $\SO(3)$ with principal orbits
$\SO(3)/\Z_2\oplus\Z_2$ and a singular orbit $\RP^2$ with normal
angle $2\pi/k$. Of the 3 functions describing the metric, one is
concave as well, and the other two are not. Thus one of the fixed
point sets of elements in $H$ is a (non-compact) surface of
revolution with positive curvature.

 \providecommand{\bysame}{\leavevmode\hbox
to3em{\hrulefill}\thinspace}

\psfrag{f_1}{$\scriptstyle f_1$}\psfrag{f_2}{$\scriptstyle
f_2$}\psfrag{f_3}{$\scriptstyle f_3$} \psfrag{g_1}{$\scriptstyle
g_1$}\psfrag{g_2}{$\scriptstyle g_2$}\psfrag{g_3}{$\scriptstyle
g_3$} \psfrag{h_1}{$\scriptstyle h_1$}\psfrag{h_2}{$\scriptstyle
h_2$}\psfrag{h_3}{$\scriptstyle h_3$}
\psfrag{f_1=f_2=f_3}{$\scriptstyle f_1=f_2=f_3$}\psfrag{t}{$$}
\begin{figure}[h]
\begin{center}
\includegraphics[width=3.1in,height=3in,angle=-90]{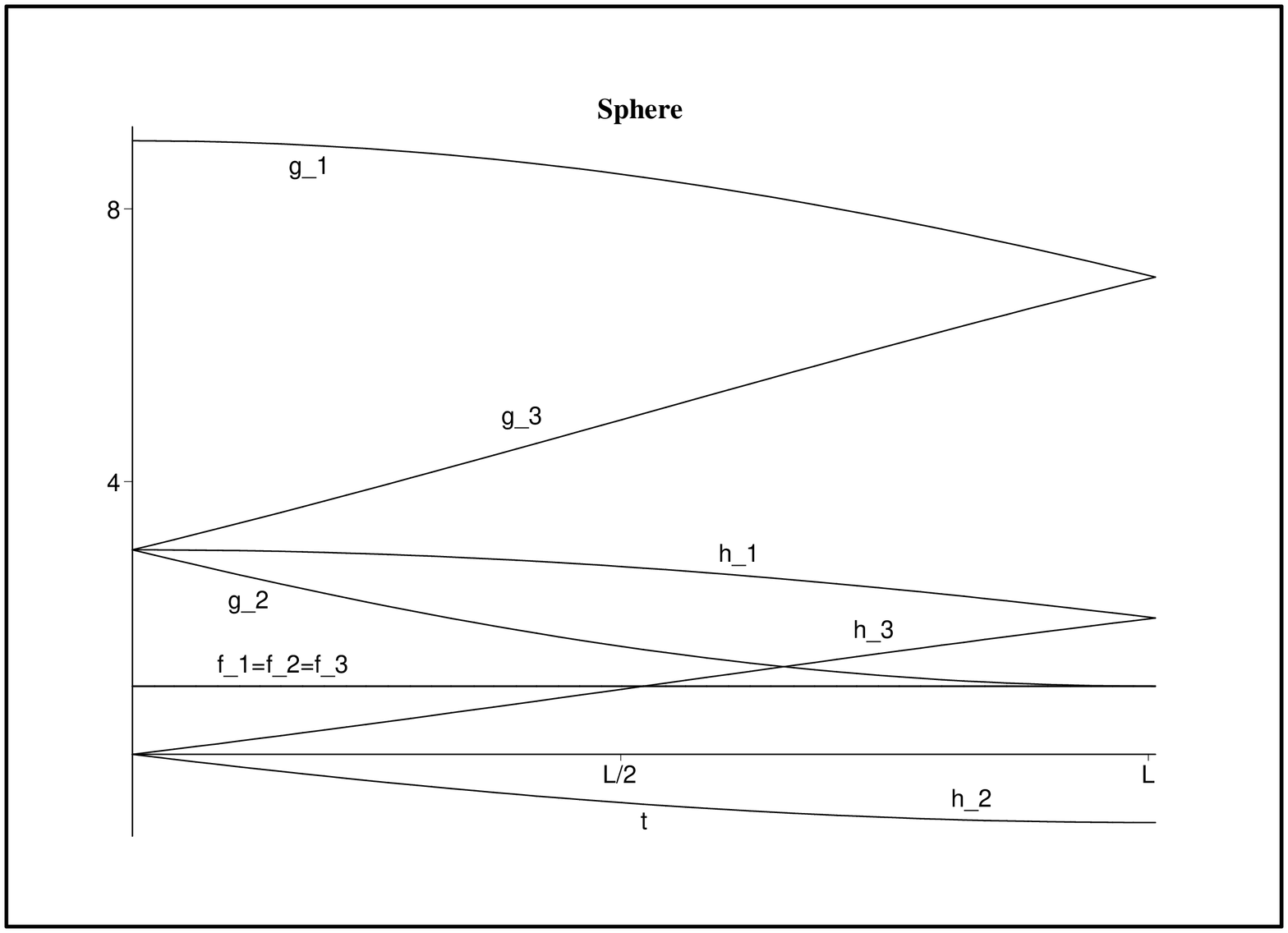}\qquad
\includegraphics[width=3.1in,height=3in,angle=-90]{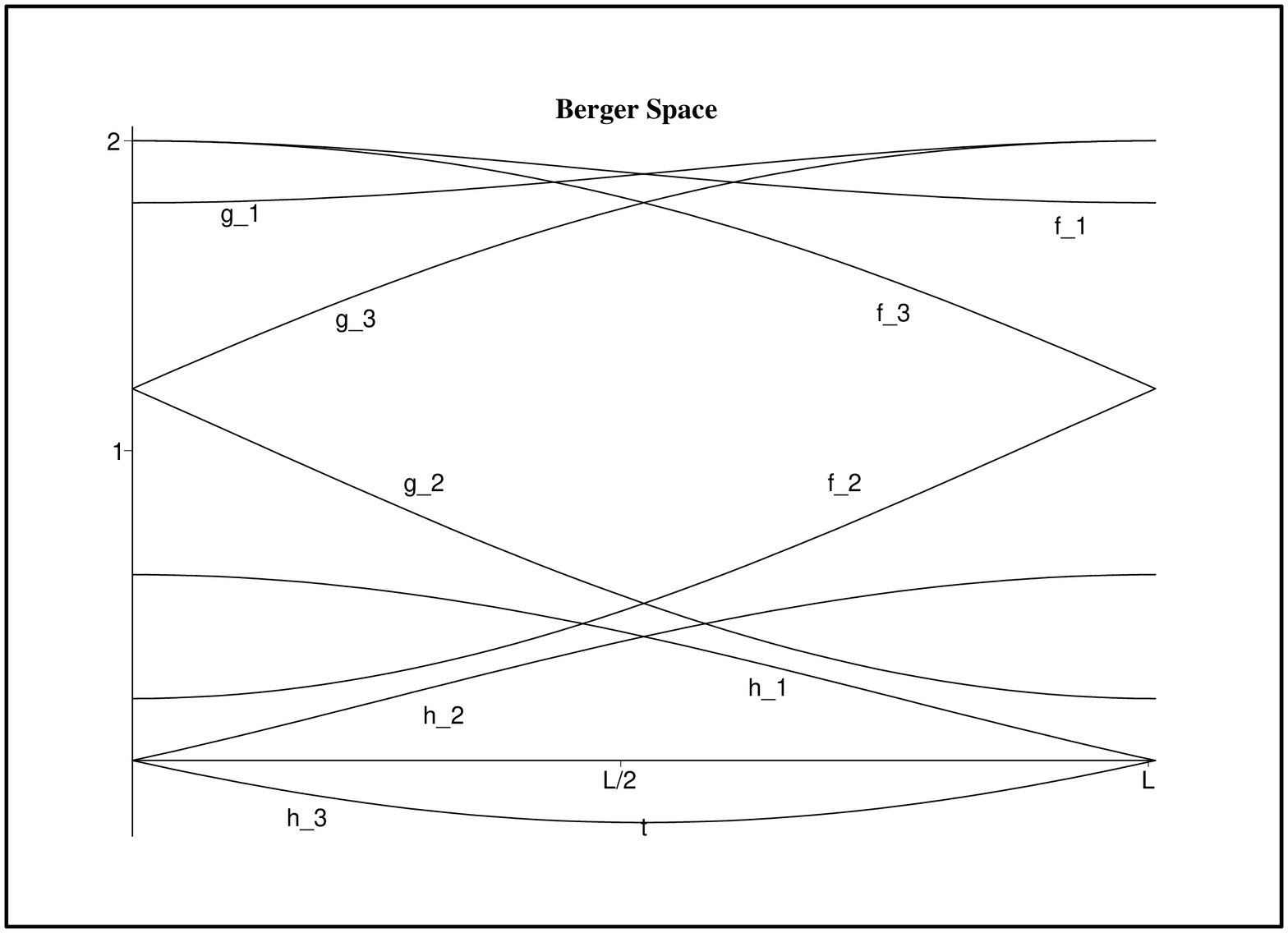}
\end{center}
\caption{All 9 functions on $[0,L]$.}
\end{figure}

\begin{figure}[h]
\begin{center}
\includegraphics[width=3.1in,height=3in,angle=-90]{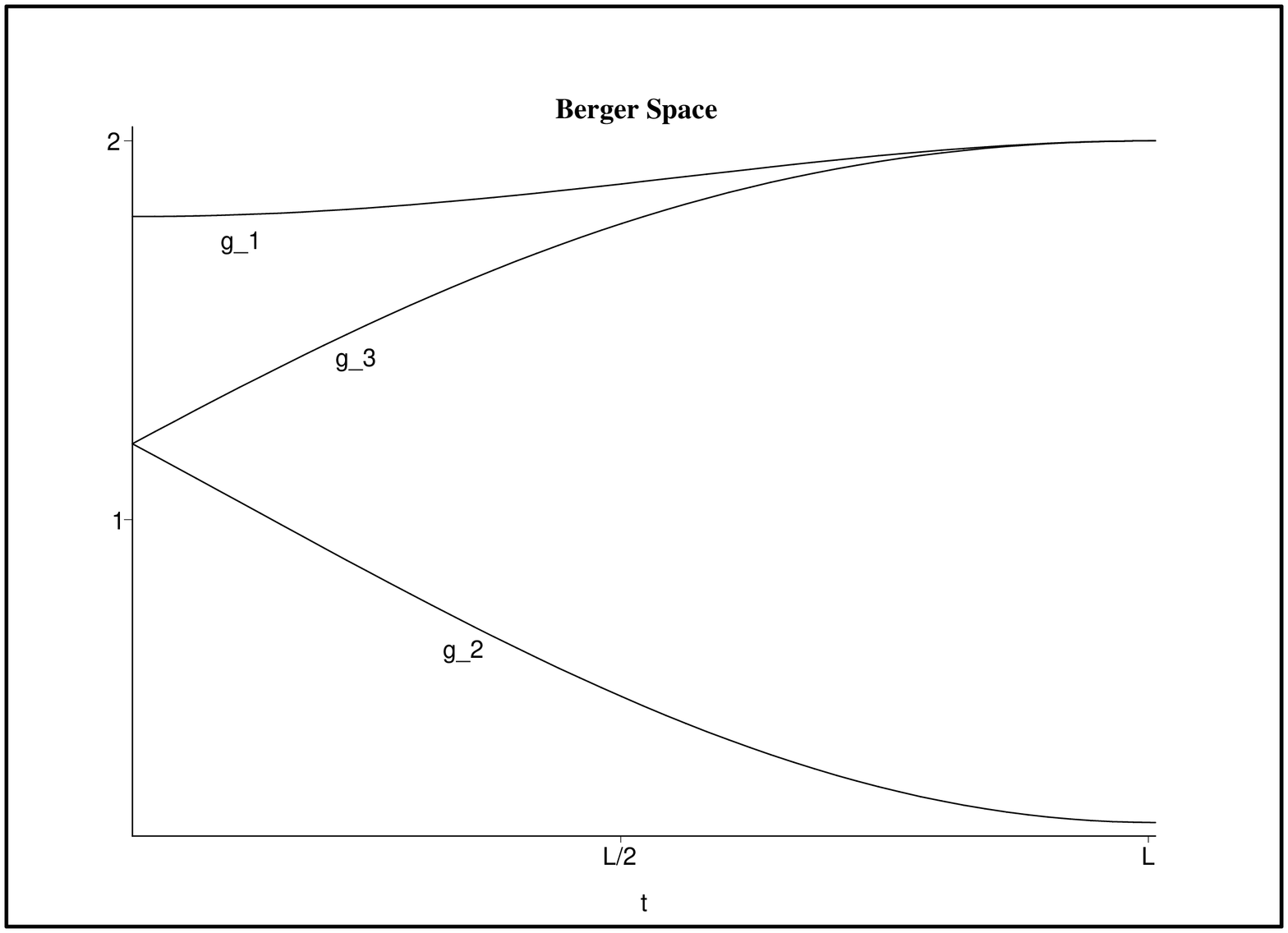}\qquad
\includegraphics[width=3.1in,height=3in,angle=-90]{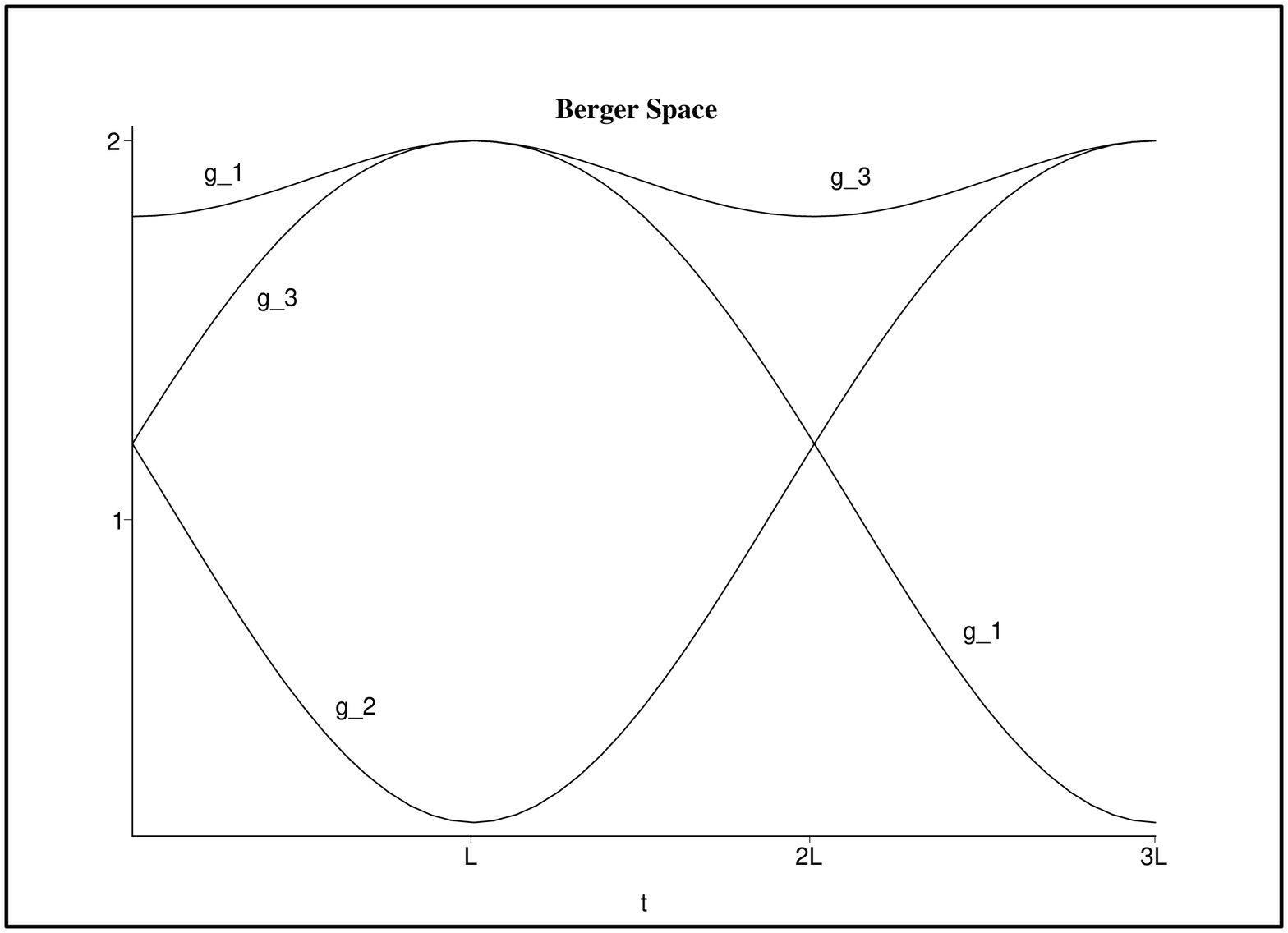}
\end{center}
\caption{The  $g$ functions on $[0,L]$ and $[0,3L]$.}
\end{figure}

\psfrag{f}{$\scriptstyle f$}\psfrag{g}{$\scriptstyle
g$}\psfrag{h}{$\scriptstyle h$}
\begin{figure}[h]
\begin{center}
\includegraphics[width=3.1in,height=3in,angle=-90]{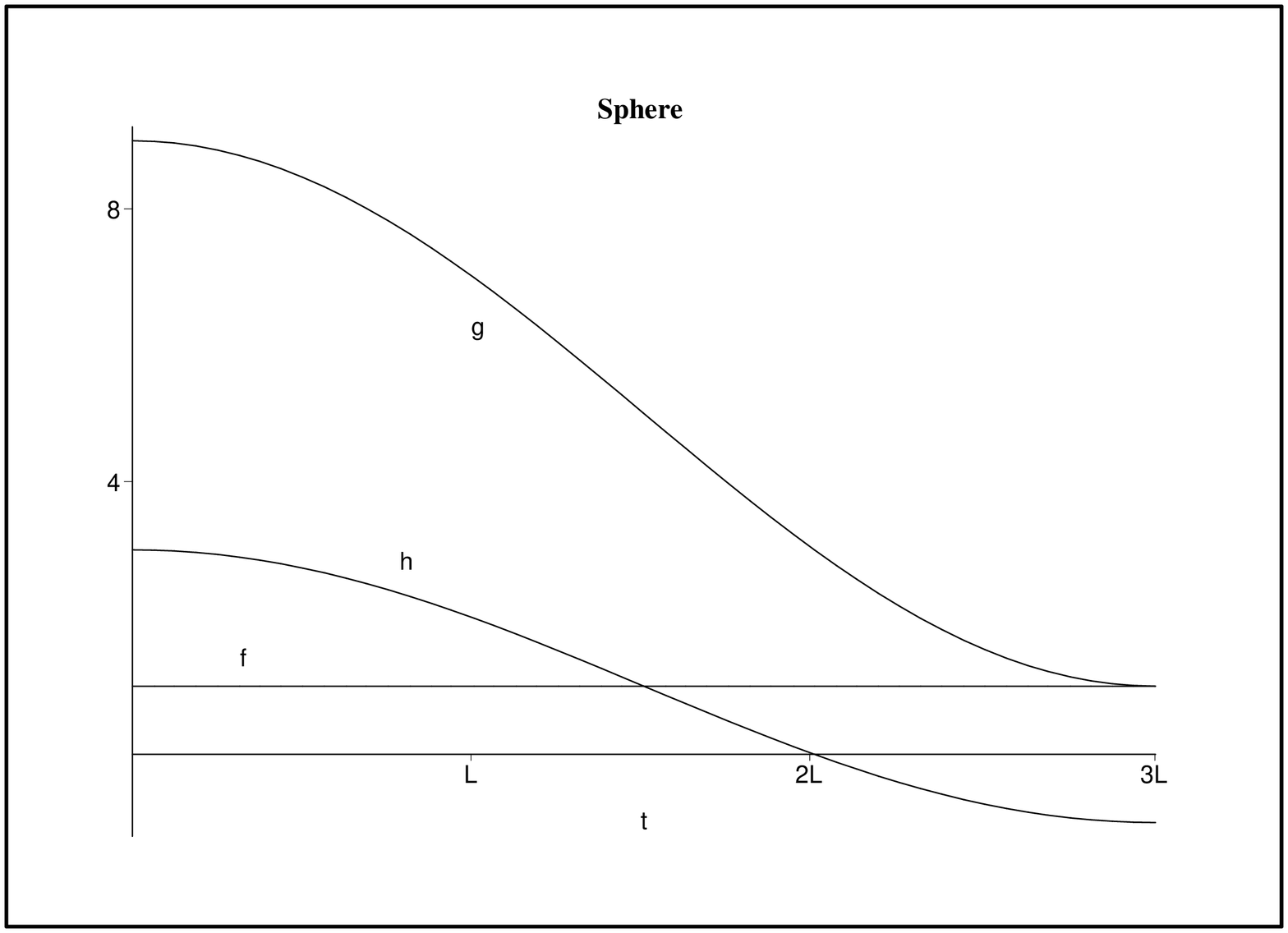}\qquad
\includegraphics[width=3.1in,height=3in,angle=-90]{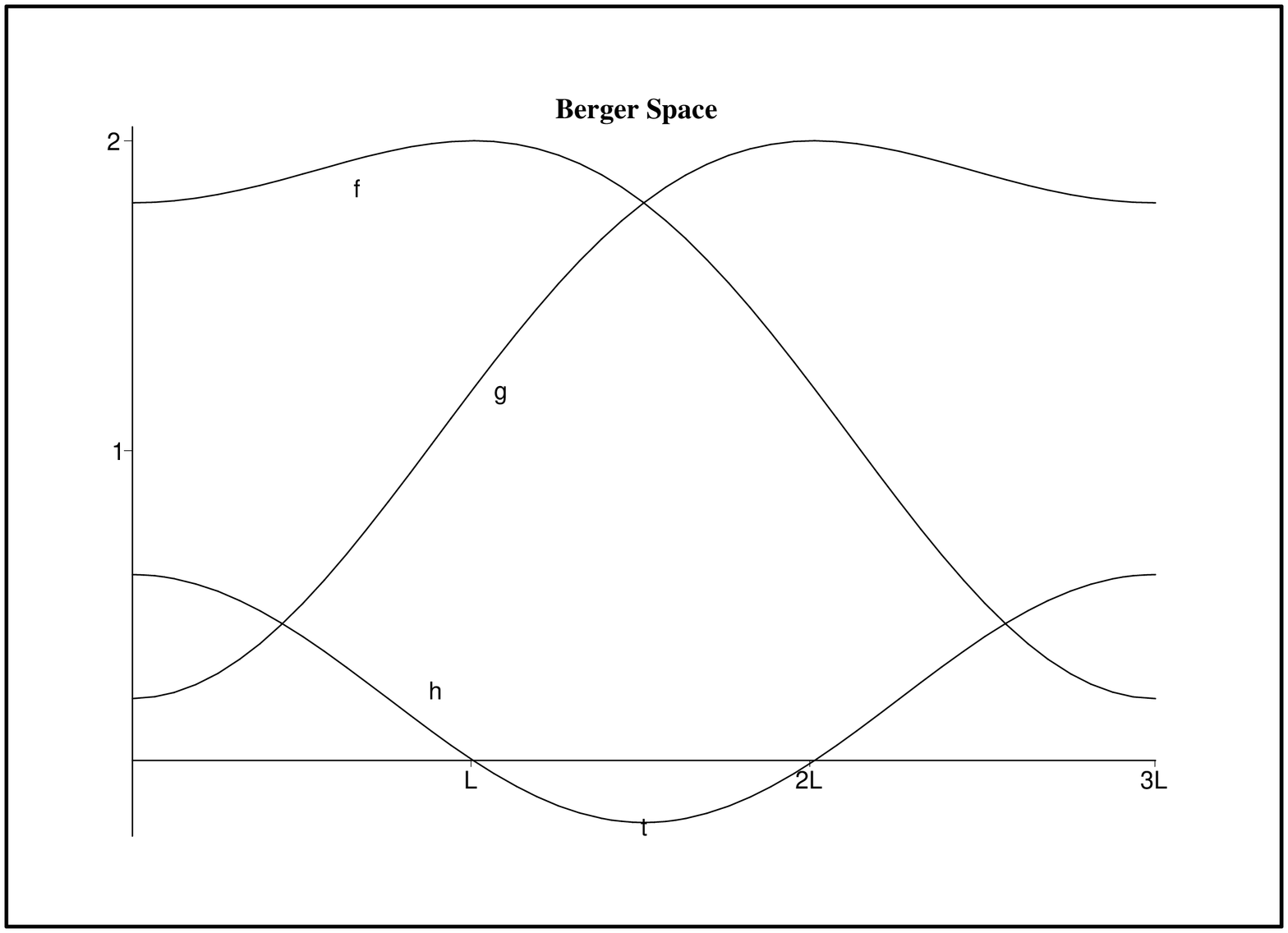}
\end{center}
\caption{  All functions on $[0,3L]$.}
\end{figure}

\psfrag{F_1}{$\scriptstyle F_1$}\psfrag{G_1}{$\scriptstyle
G_1$}\psfrag{H_1}{$\scriptstyle H_1$}
\begin{figure}[h]
\begin{center}
\includegraphics[width=3.1in,height=3in,angle=-90]{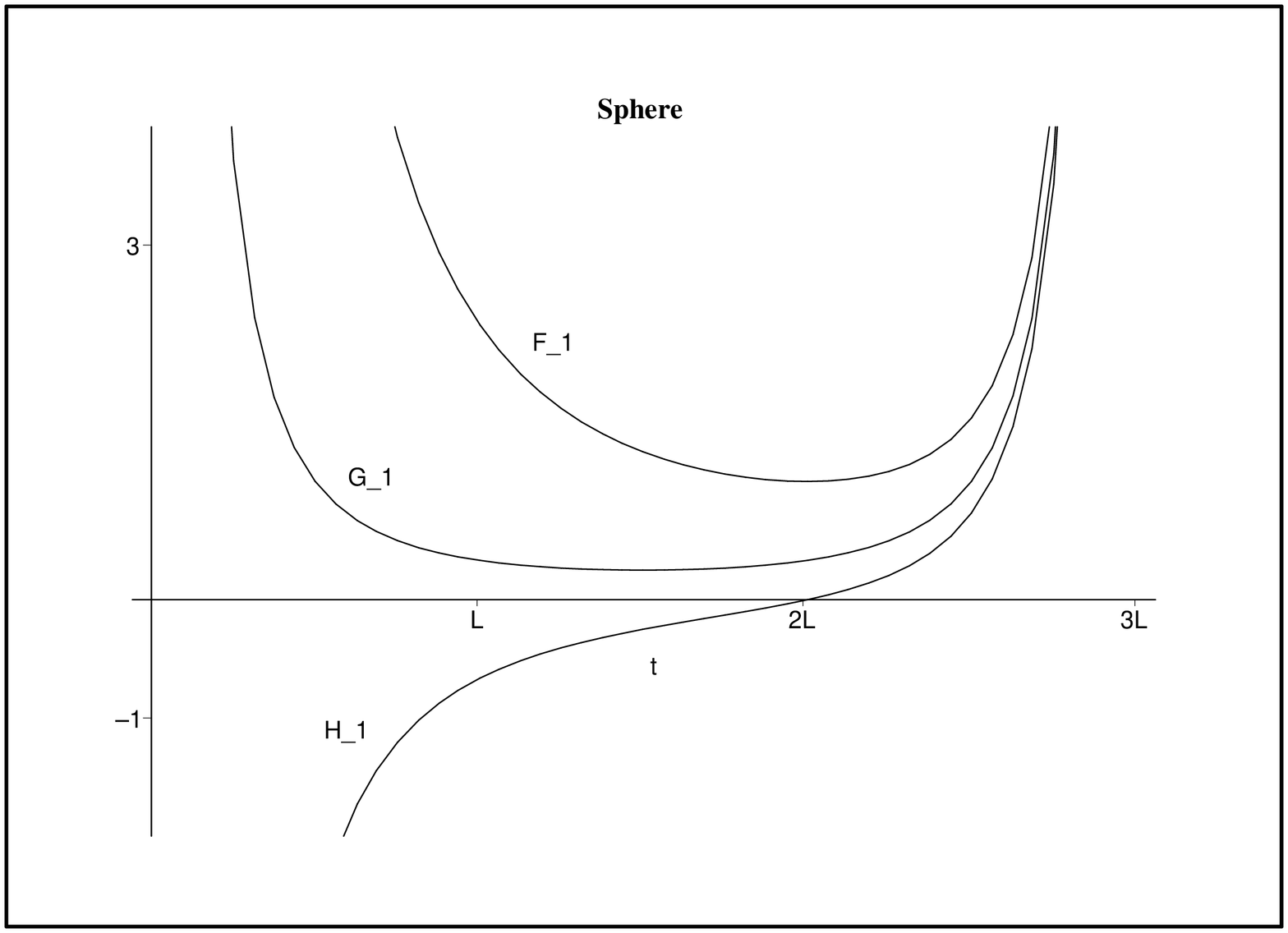}\qquad
\includegraphics[width=3.1in,height=3in,angle=-90]{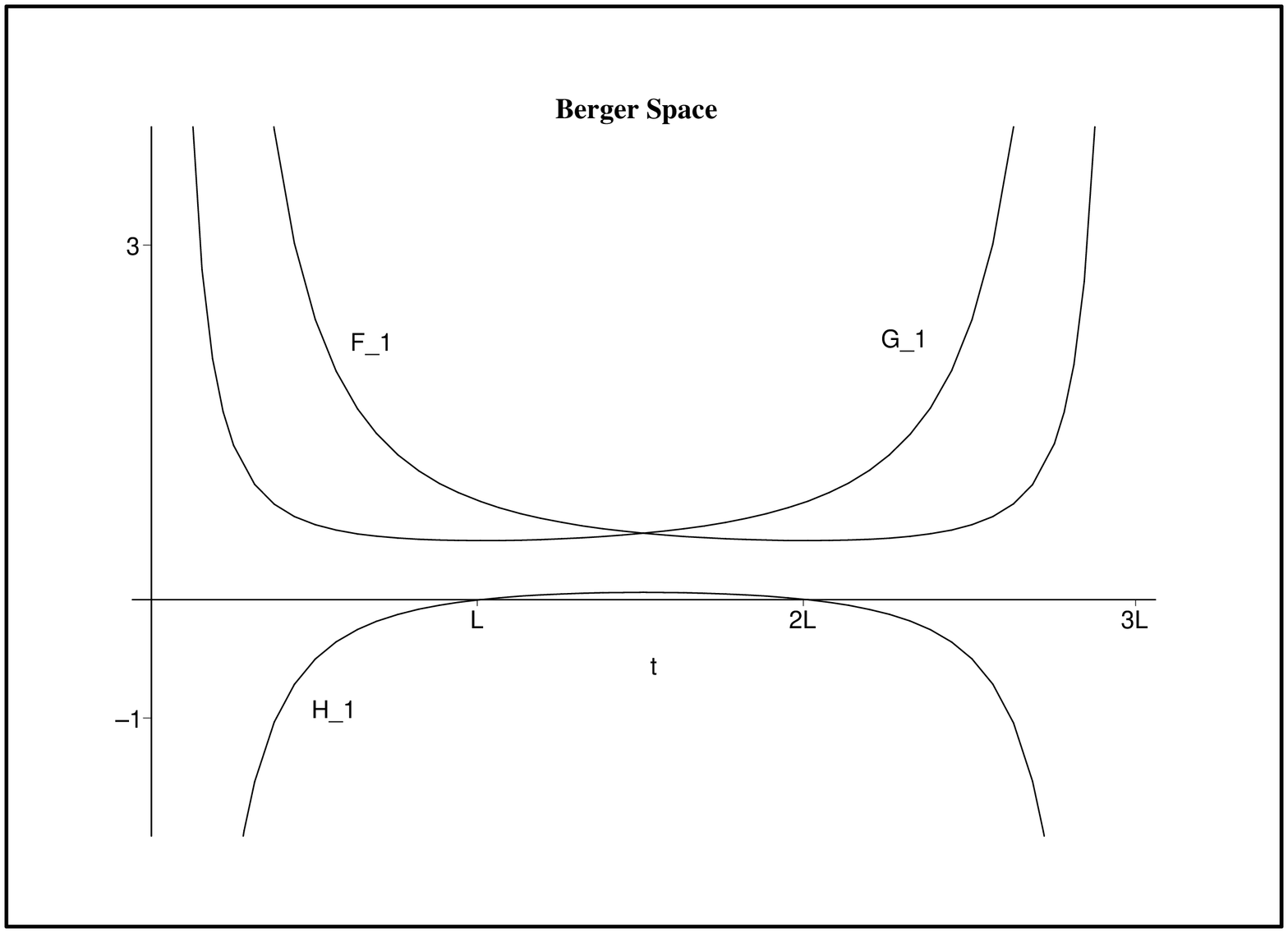}
\end{center}
\caption{The  inverse functions on $[0,3L]$.}
\end{figure}

\psfrag{f_1}{$\scriptstyle f_1$}\psfrag{f_2}{$\scriptstyle
f_2$}\psfrag{f_3}{$\scriptstyle f_3$}
 \psfrag{h_1}{$\scriptstyle
h_1$}\psfrag{h_2=h3}{$\scriptstyle
h_2=h_3$}\psfrag{g_2=g_3}{$\scriptstyle g_2=g_3$}
\psfrag{g_1=g_2=g_3}{$\scriptstyle g_1=g_2=g_3$}
\psfrag{f_2=f_3}{$\scriptstyle f_2=f_3$} \psfrag{g_1}{$\scriptstyle
g_1$}
\begin{figure}[h]
\begin{center}
\includegraphics[width=3.1in,height=3in,angle=-90]{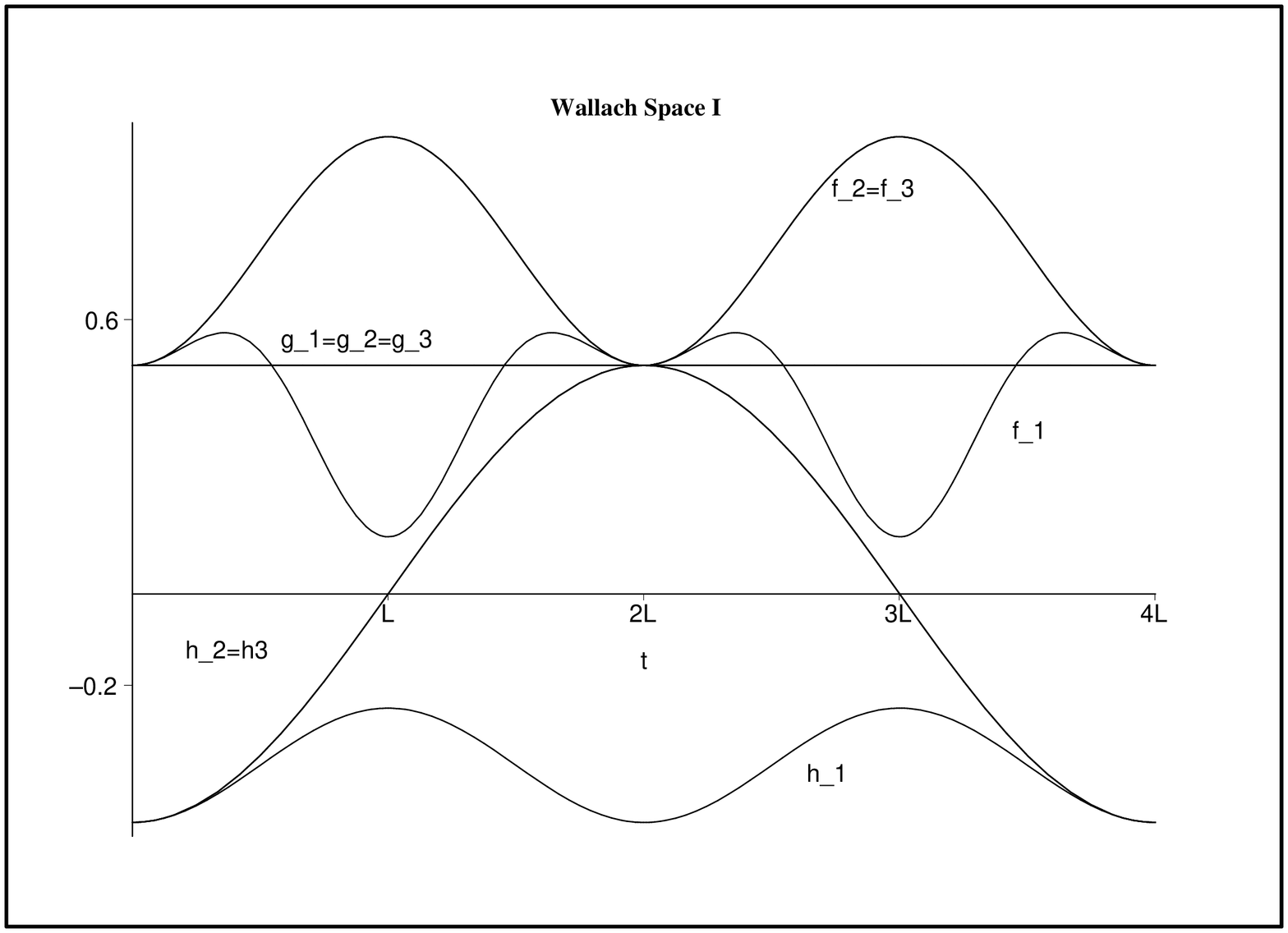}\qquad
\includegraphics[width=3.1in,height=3in,angle=-90]{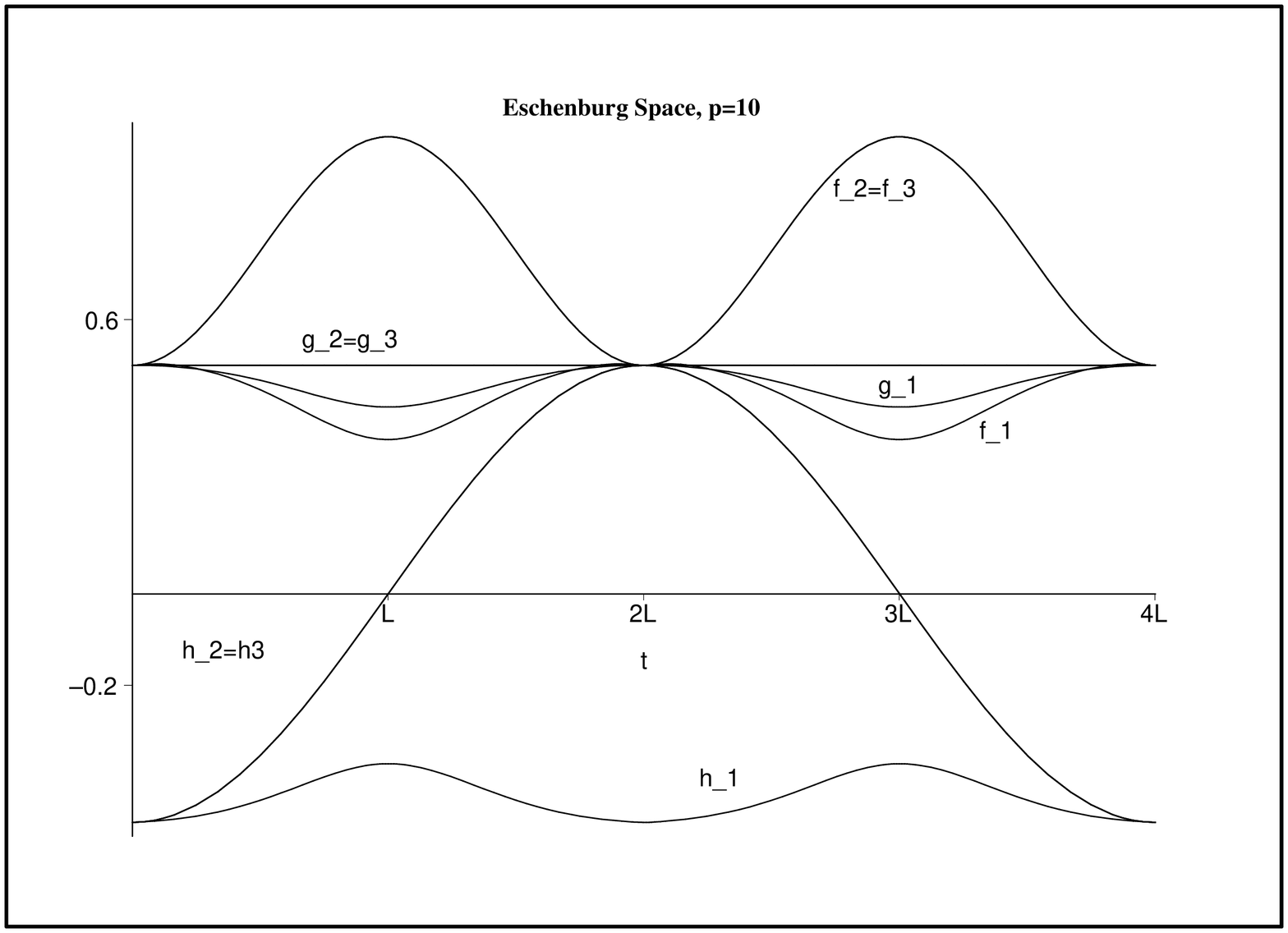}
\end{center}
\caption{Wallach space $W^7_{(1)}$ and Eschenburg space $E_{10}$ on
$[0,4L]$.}
\end{figure}

\begin{figure}[h]
\begin{center}
\includegraphics[width=3.1in,height=3in,angle=-90]{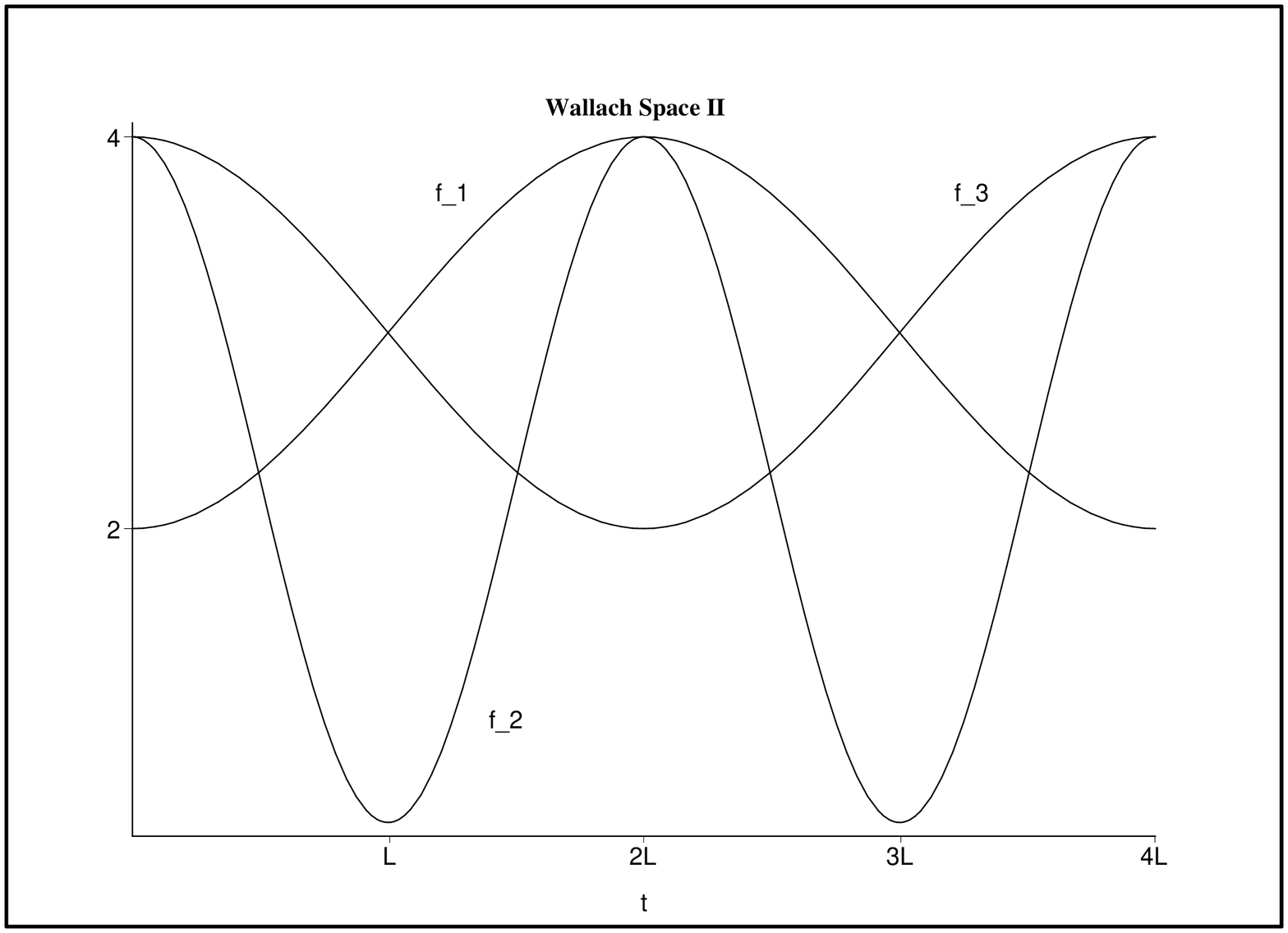}\qquad
\includegraphics[width=3.1in,height=3in,angle=-90]{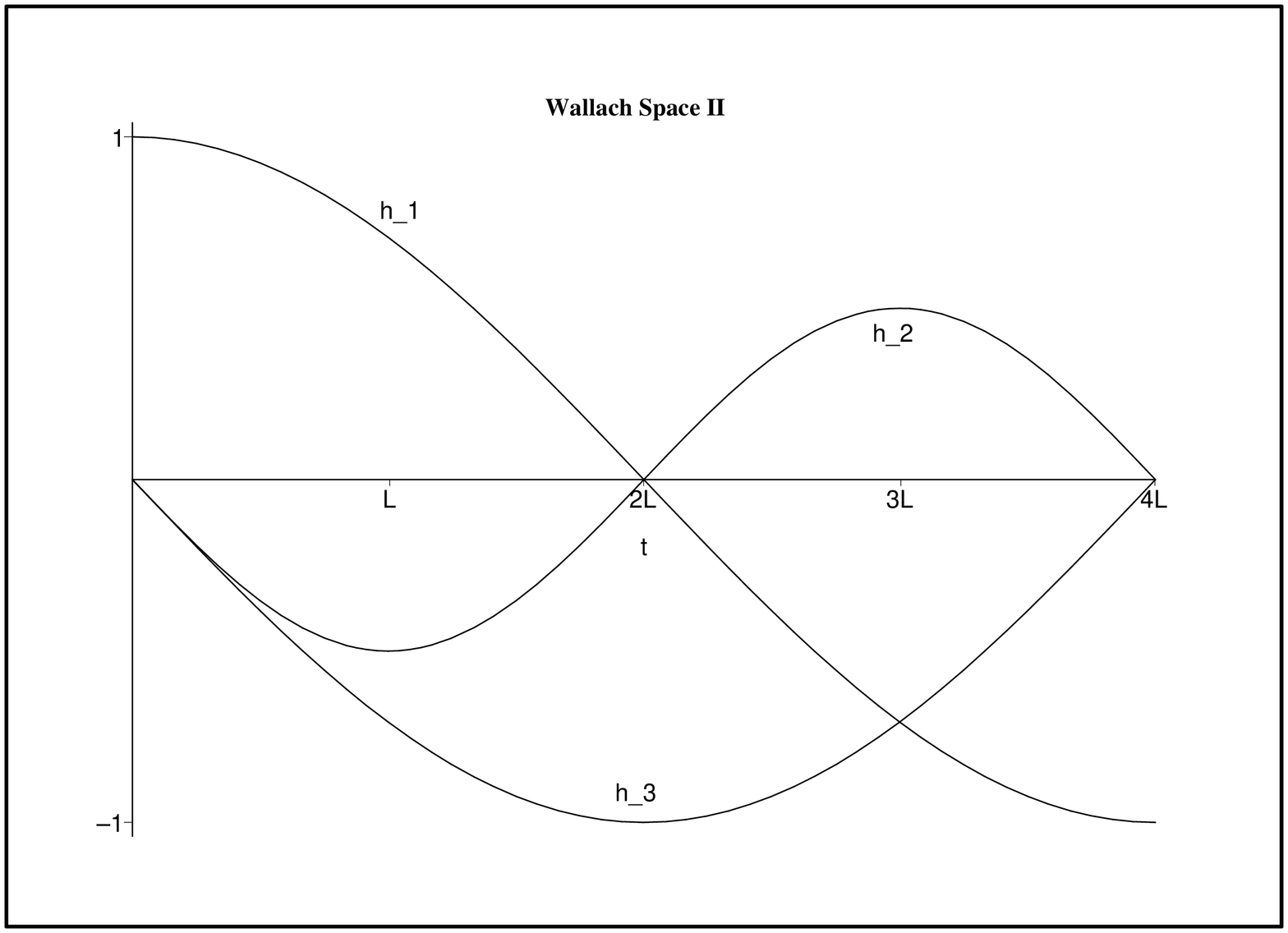}
\end{center}
\caption{Wallach space $W^7_{(2)}$  on $[0,4L]$.}
\end{figure}

\begin{figure}[h]
\begin{center}
\includegraphics[width=3.1in,height=3in,angle=-90]{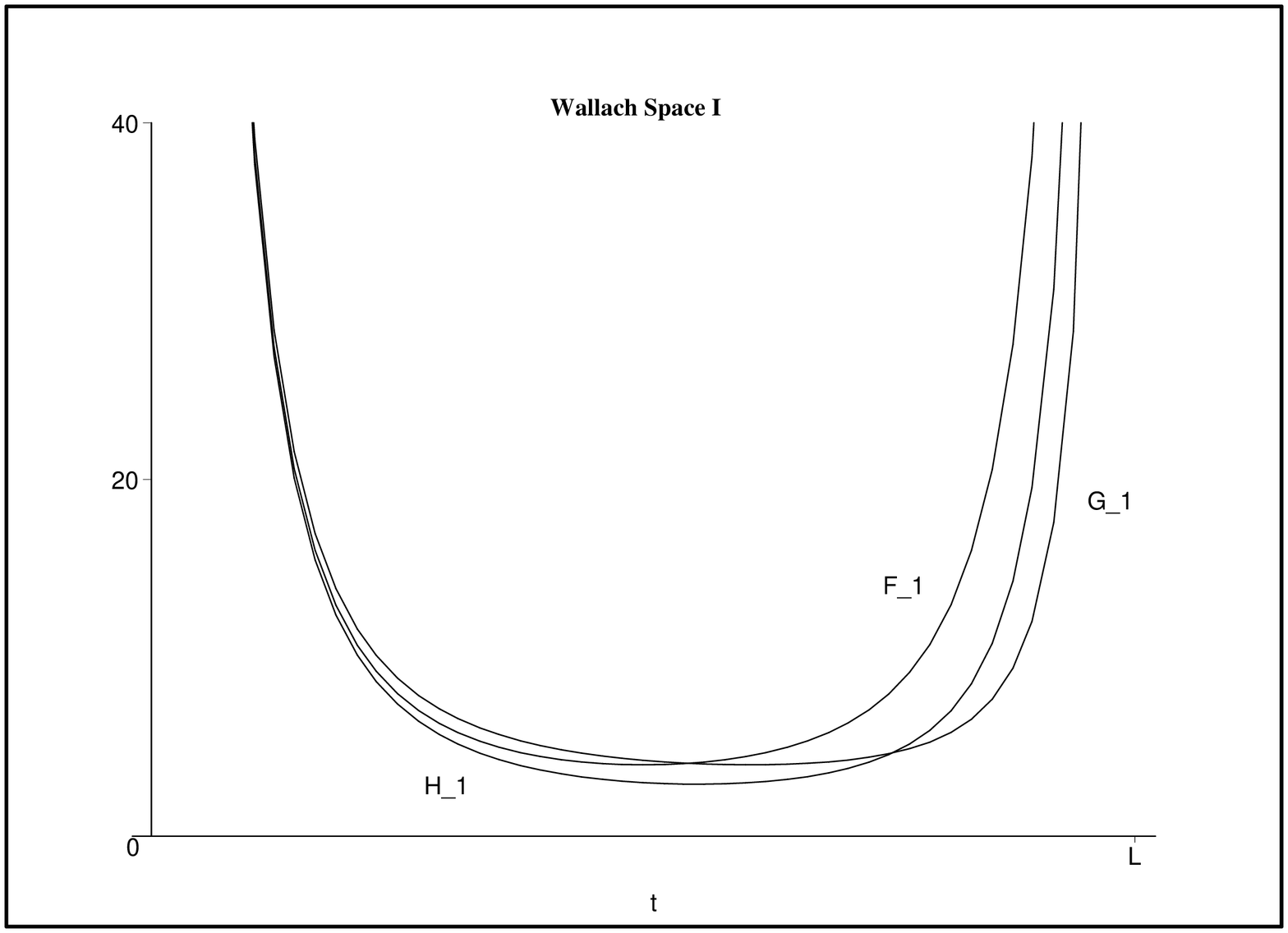}\qquad
\includegraphics[width=3.1in,height=3in,angle=-90]{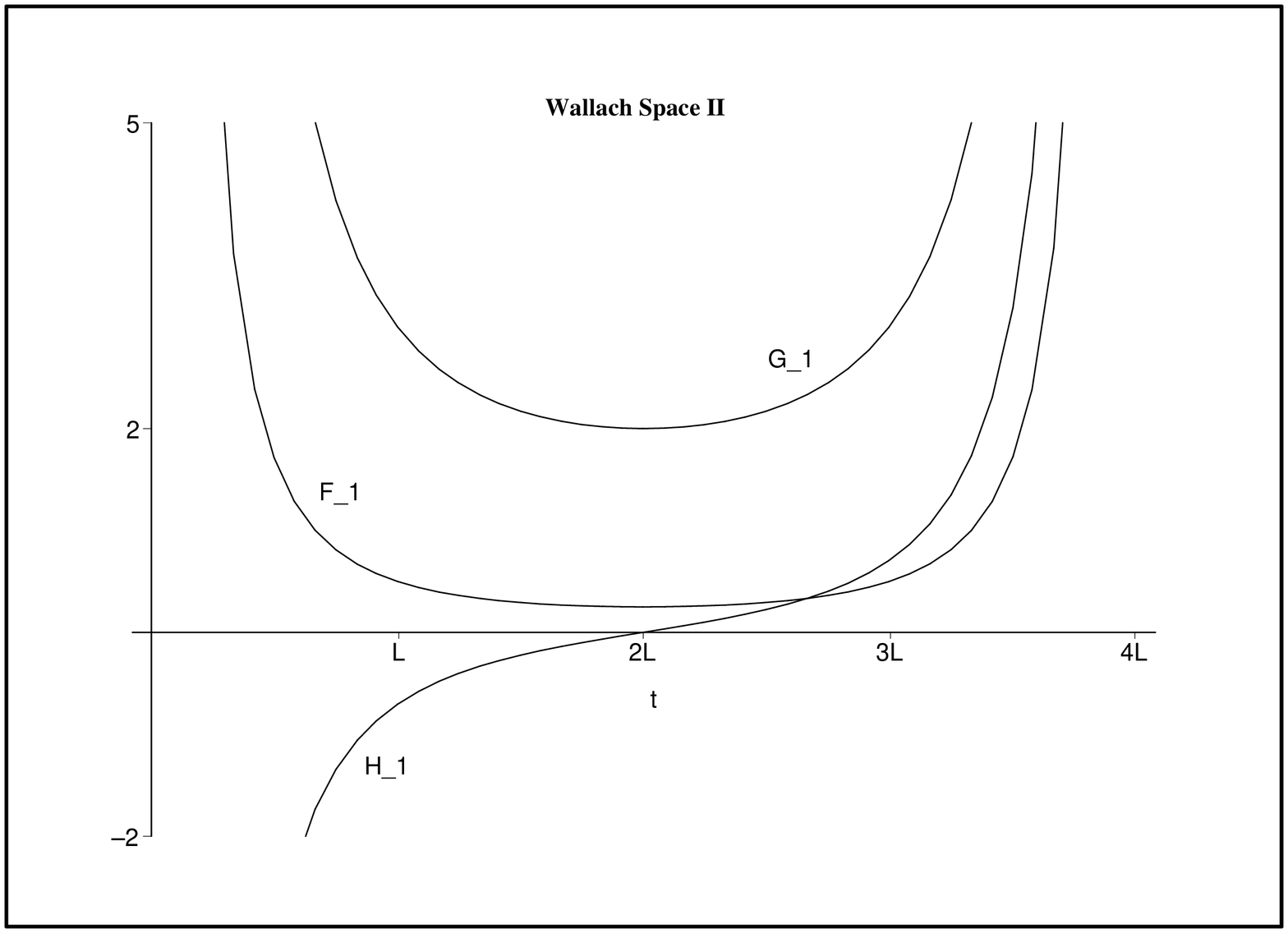}
\end{center}
\caption{Inverse functions for $W^7_{(1)}$  and  $W^7_{(2)}$.}
\end{figure}

\psfrag{f1}{$ \scriptstyle f_1$}\psfrag{f2}{$\scriptstyle
f_2$}\psfrag{f3}{$\scriptstyle f_3$}
\psfrag{sec(T,X1)}{$\scriptstyle
\sec(\gamma',X_1^*)$}\psfrag{sec(T,X2)}{$\scriptstyle
\sec(\gamma',X_2^*)$}\psfrag{sec(T,X3)}{$\scriptstyle
\sec(\gamma',X_3^*)$}
\begin{figure}[h]
\begin{center}
\includegraphics[width=3.1in,height=3in,angle=-90]{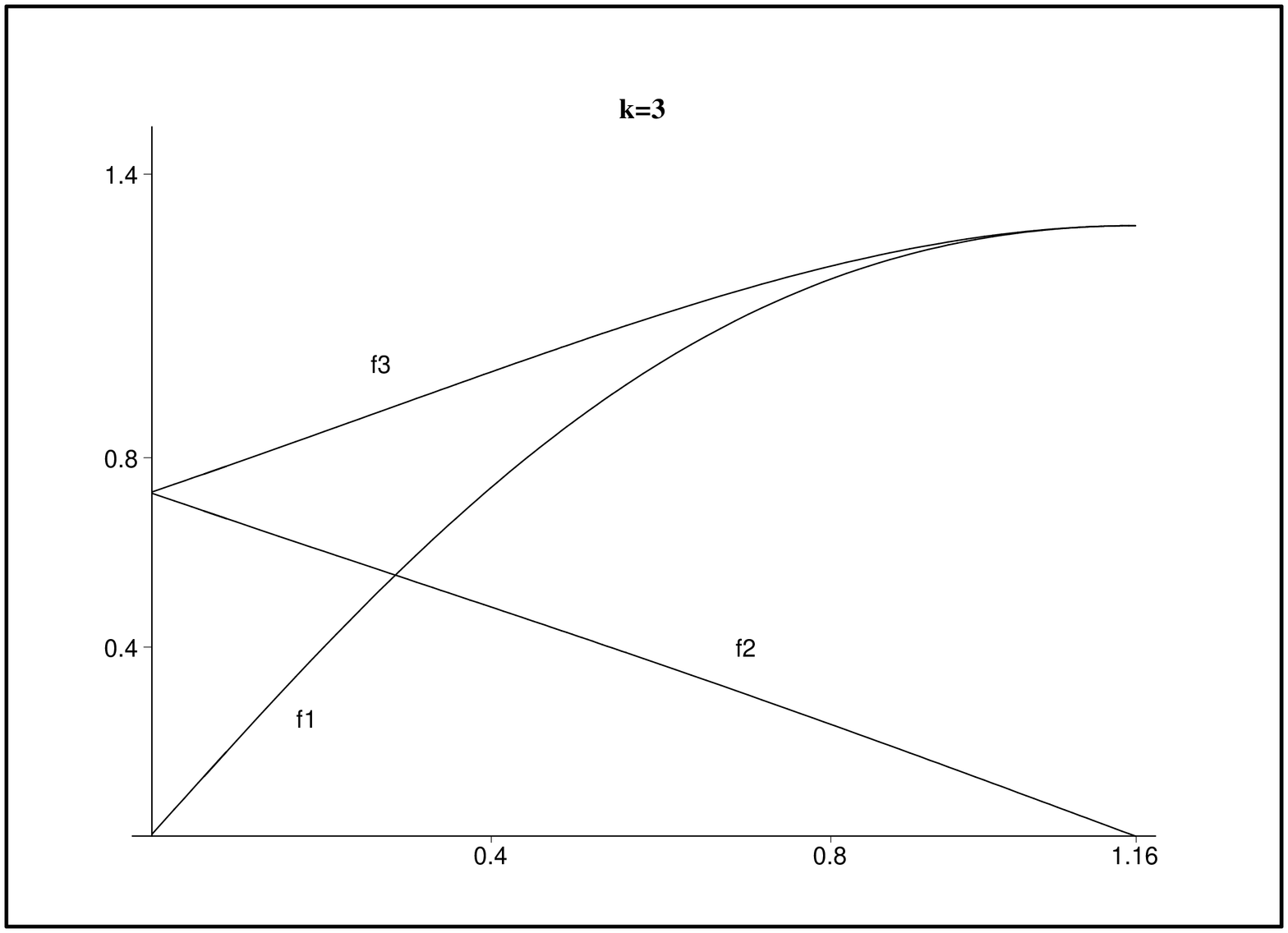}\qquad
\includegraphics[width=3.1in,height=3in,angle=-90]{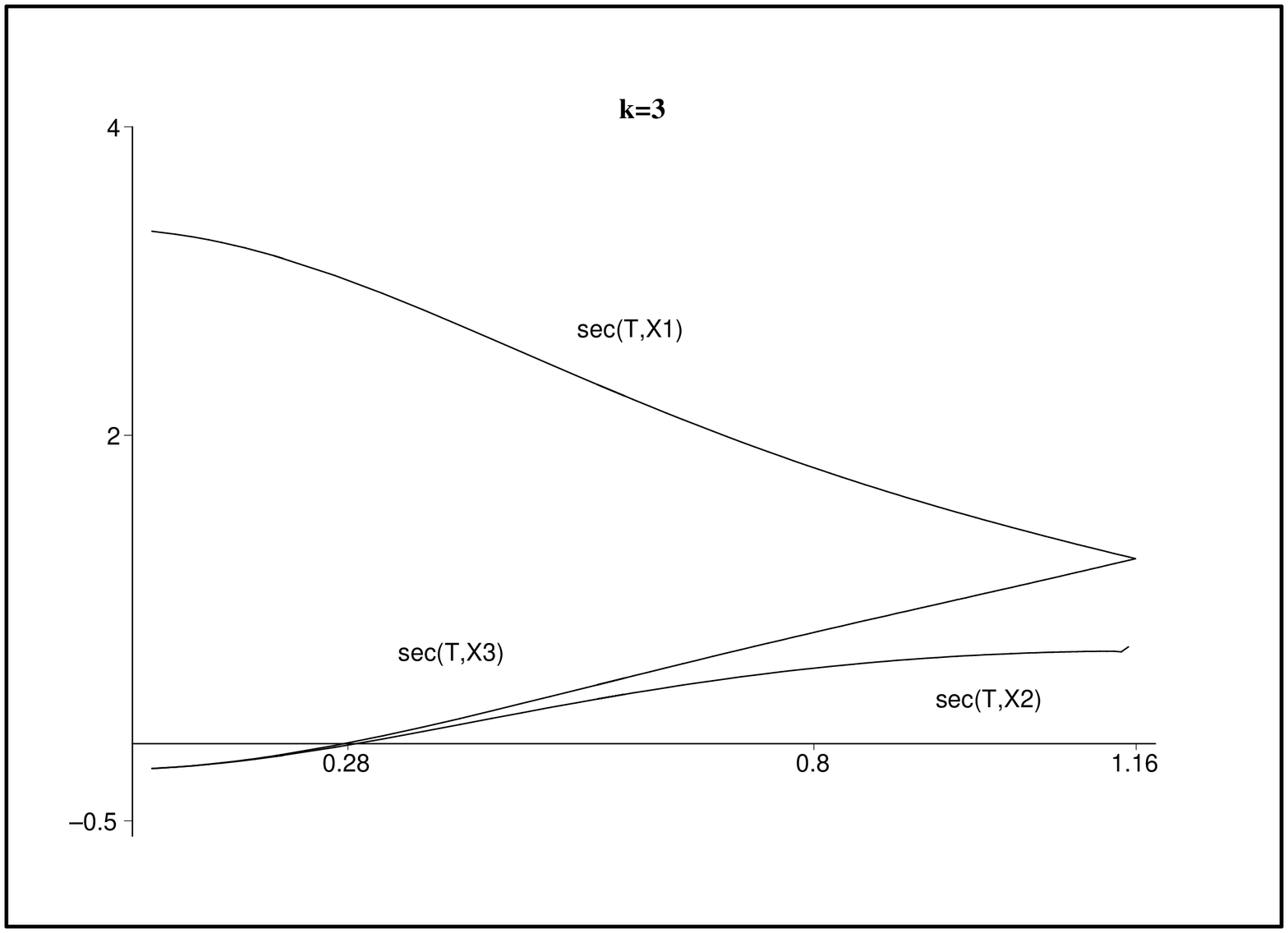}
\end{center}
\caption{Hitchin metric with Normal angle $2\pi/3$.}
\end{figure}

\begin{figure}[h]
\begin{center}
\includegraphics[width=3.1in,height=3in,angle=-90]{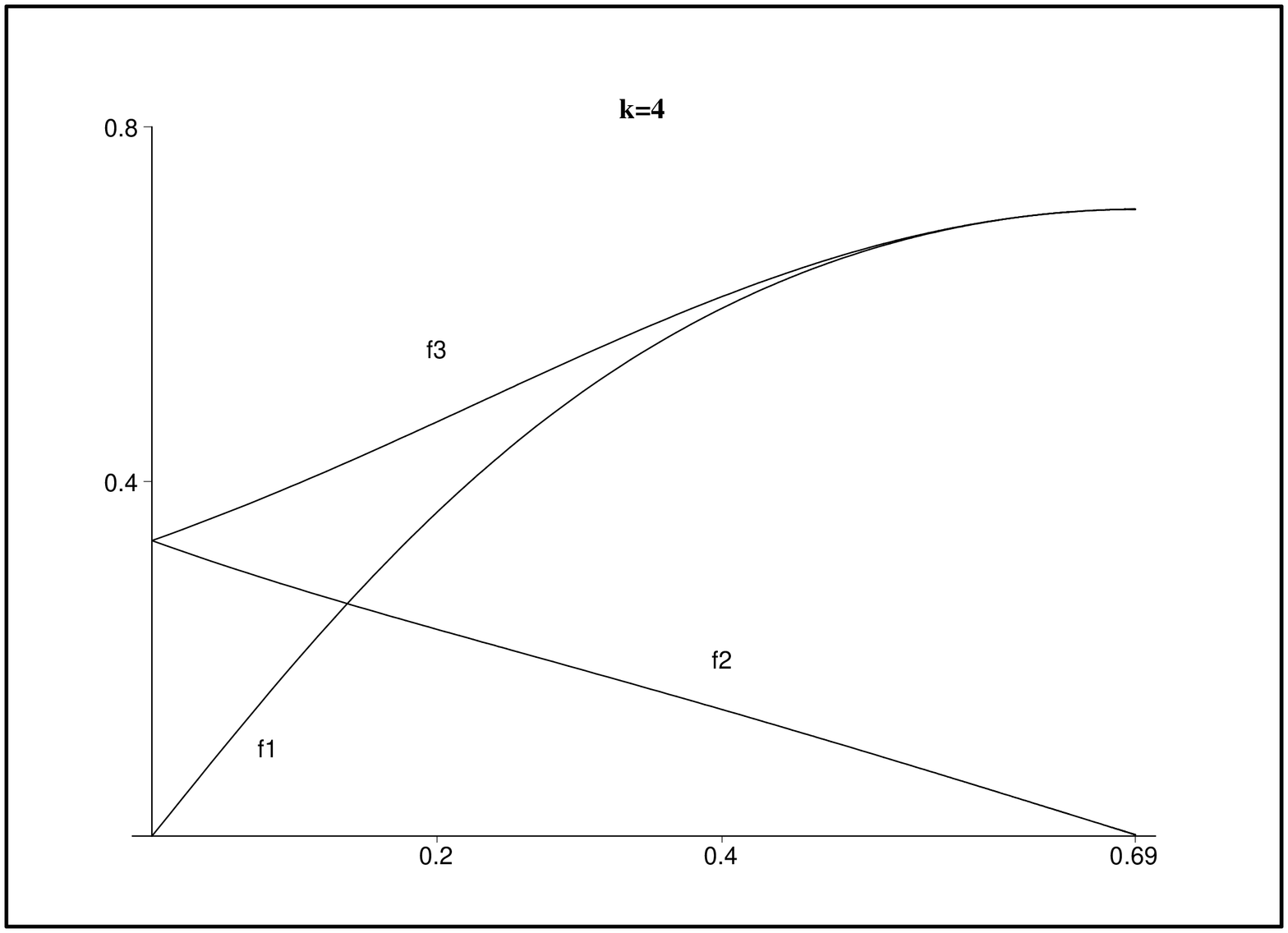}\qquad
\includegraphics[width=3.1in,height=3in,angle=-90]{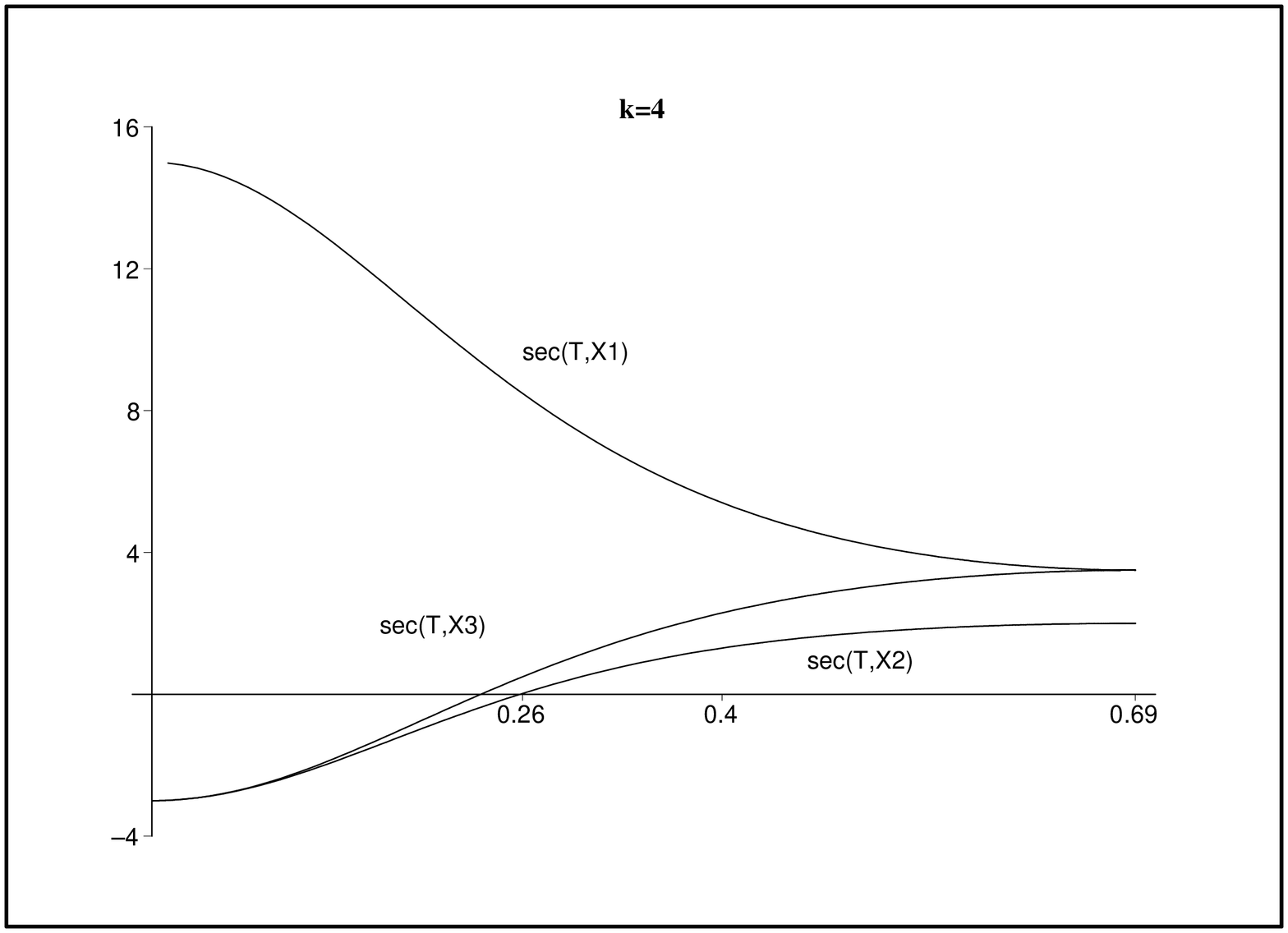}
\end{center}
\caption{Hitchin metric with Normal angle $2\pi/4$.}
\end{figure}

\begin{figure}[h]
\begin{center}
\includegraphics[width=3.1in,height=3in,angle=-90]{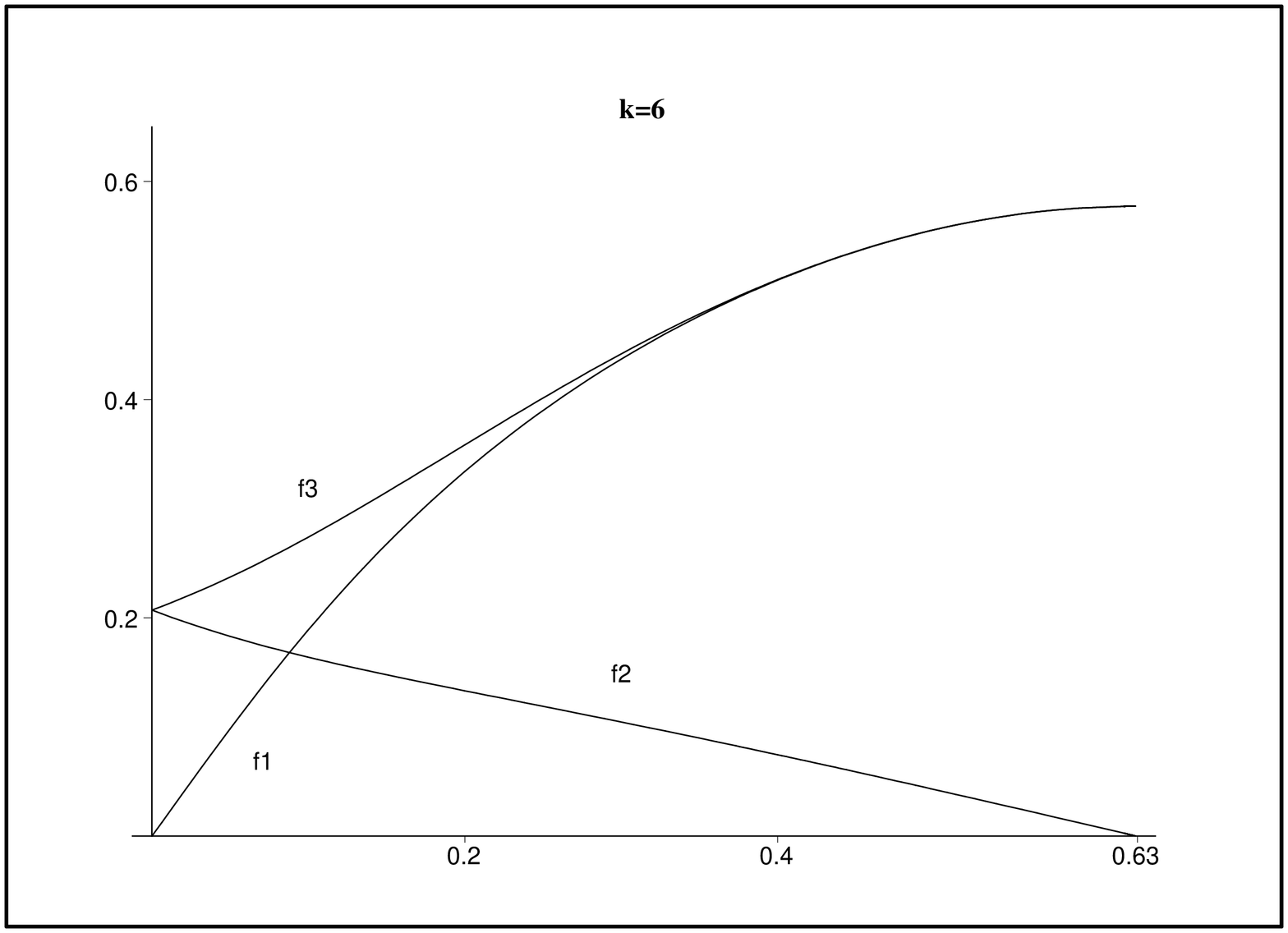}\qquad
\psfrag{X1}{$\scriptstyle X_1$}
\includegraphics[width=3.1in,height=3in,angle=-90]{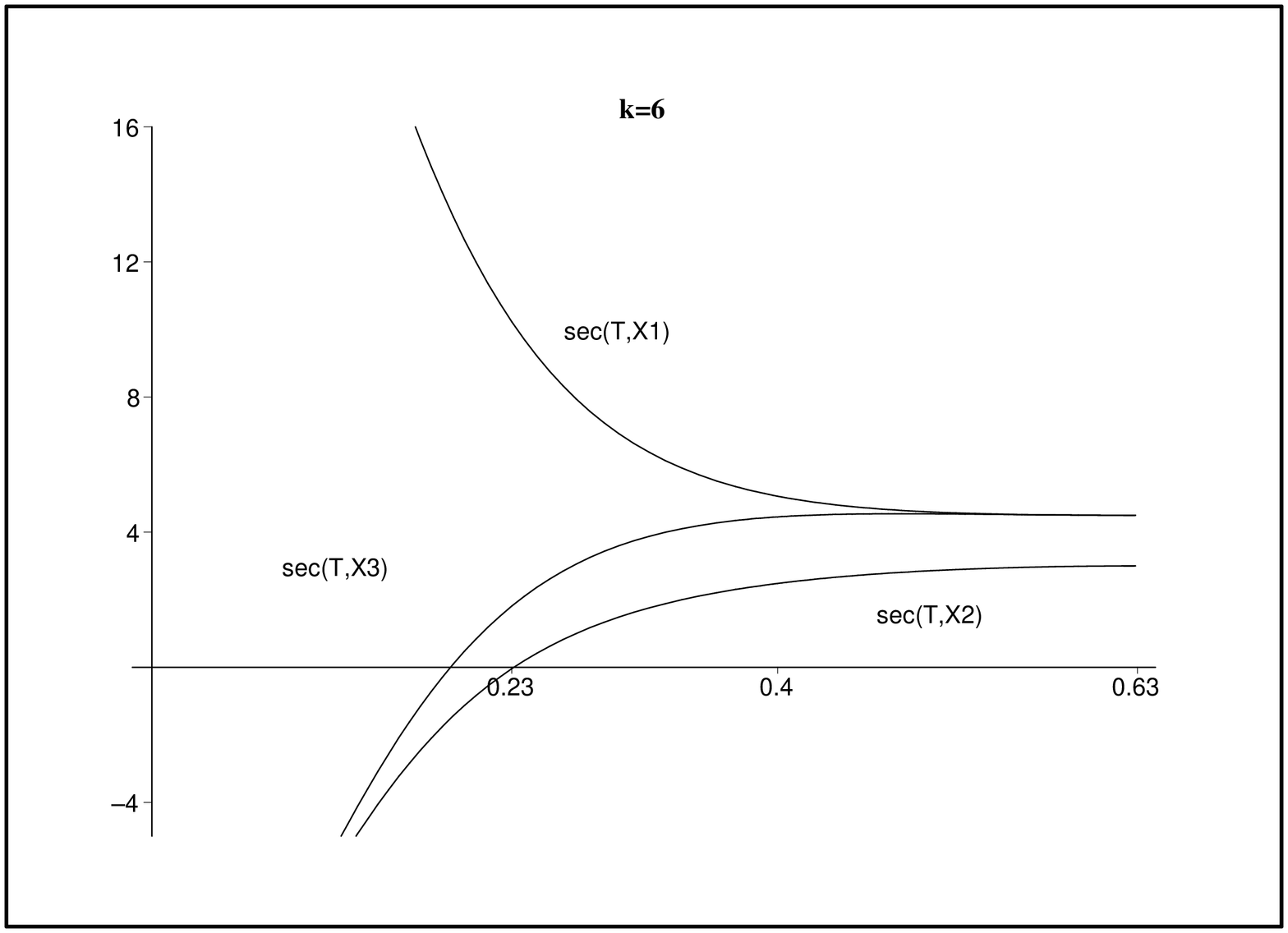}
\end{center}
\caption{Hitchin metric with Normal angle $2\pi/6$.}
\end{figure}

\begin{figure}[h]
\begin{center}
\includegraphics[width=3.1in,height=3in,angle=-90]{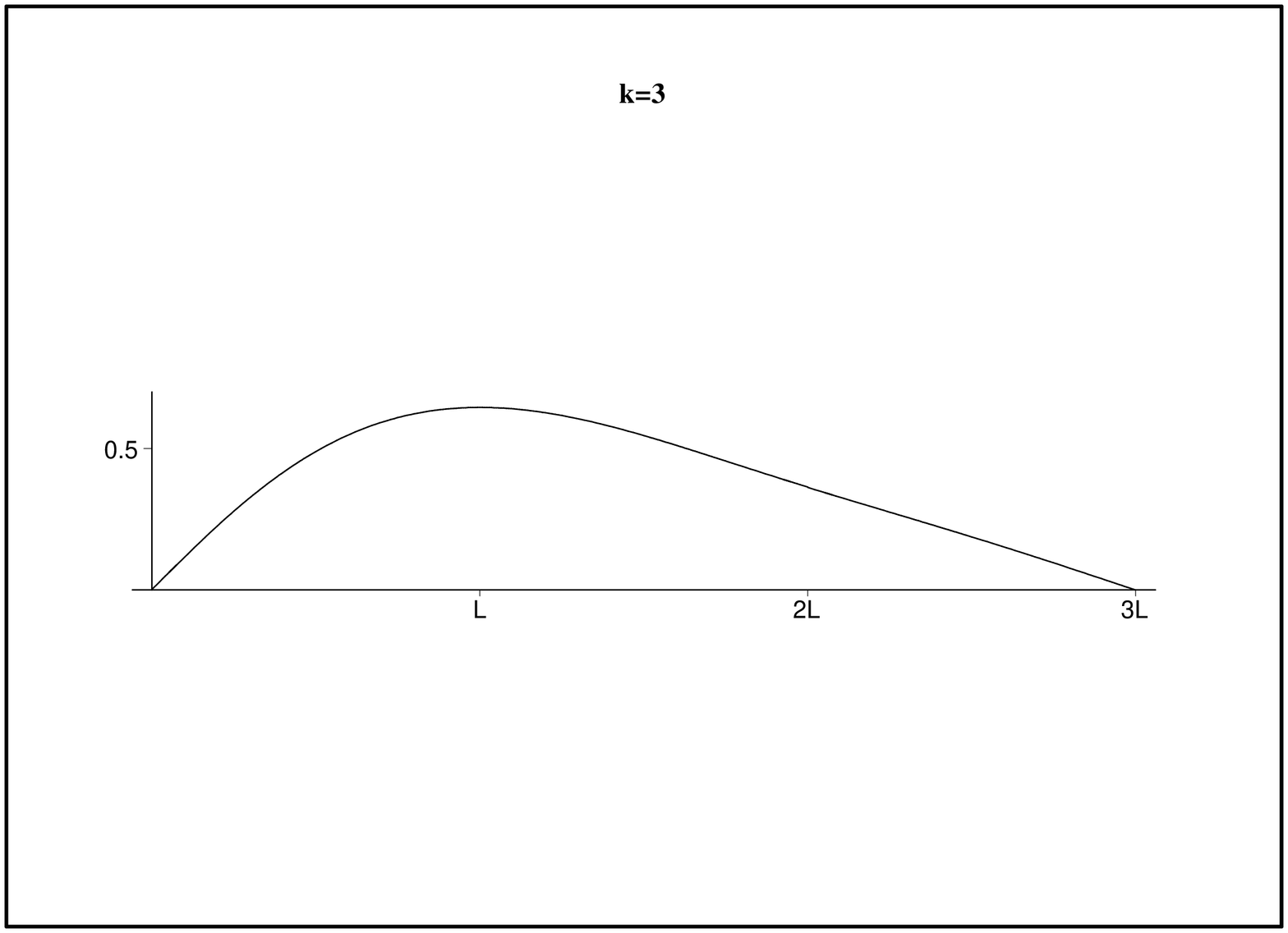}\qquad
\includegraphics[width=3.1in,height=3in,angle=-90]{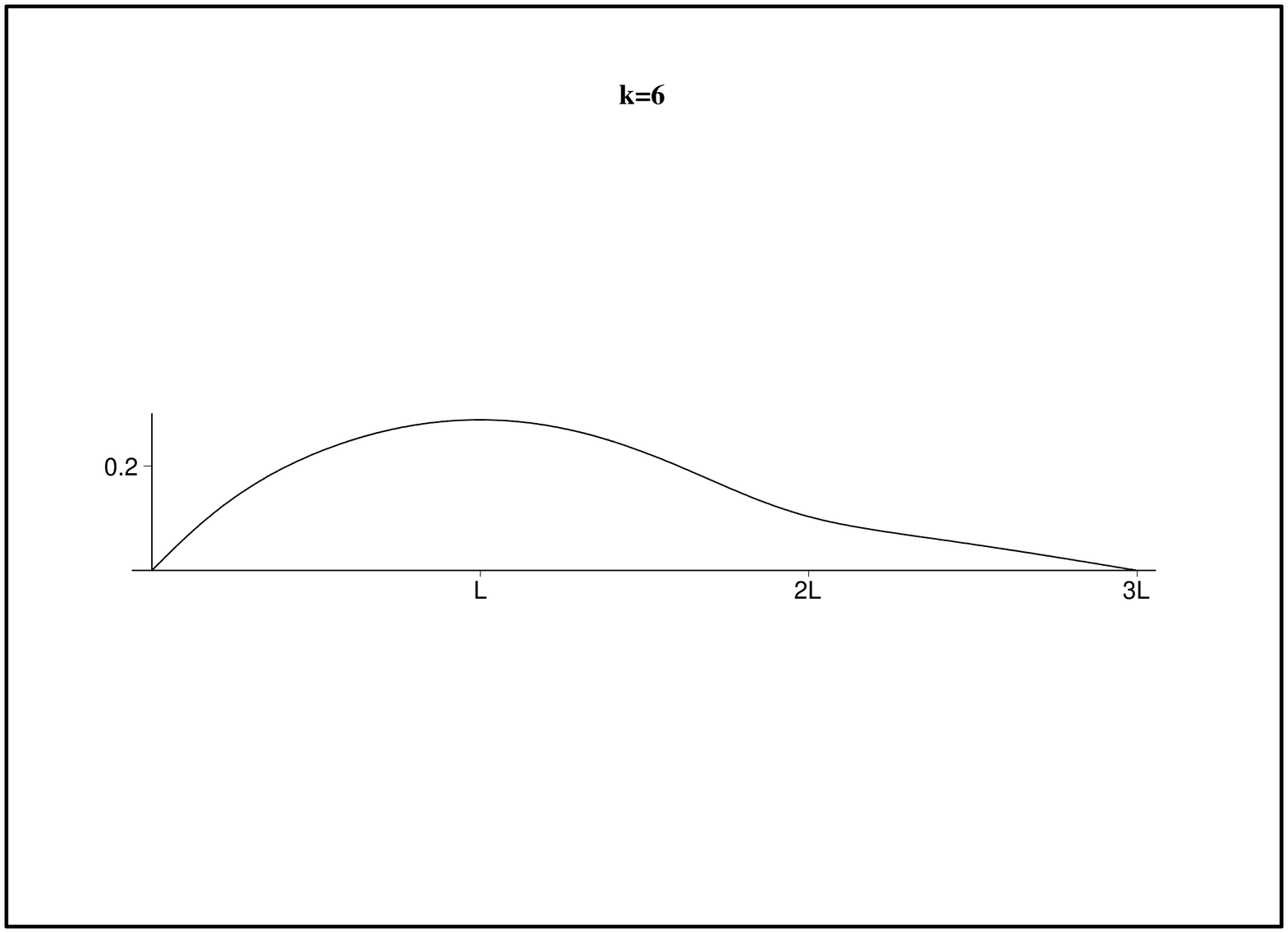}
\end{center}
\caption{Hitchin metrics on $[0,3L]$.}
\end{figure}

\begin{figure}[h]
\begin{center}
\includegraphics[width=3.1in,height=3in,angle=-90]{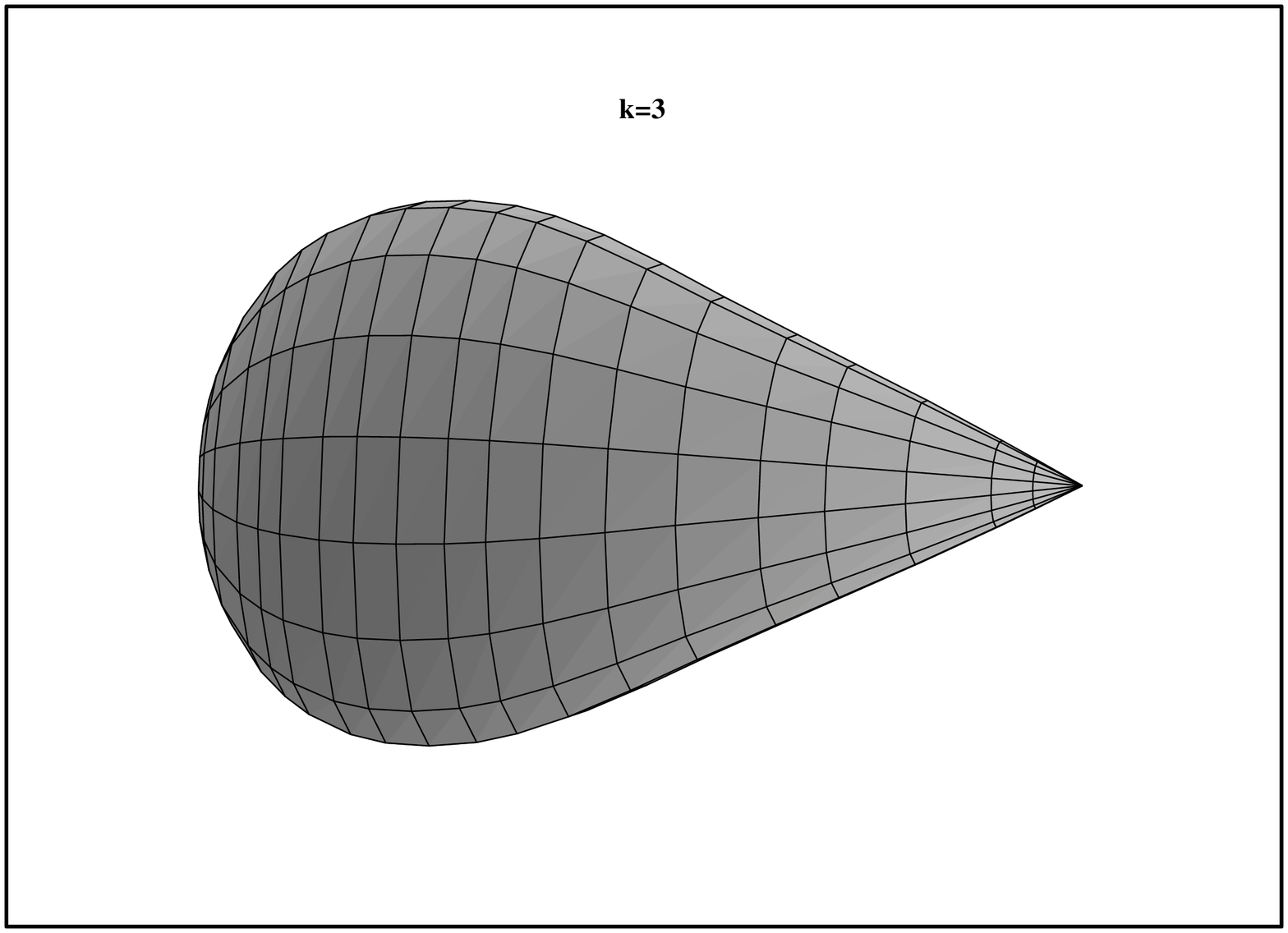}\qquad
\includegraphics[width=3.1in,height=3in,angle=-90]{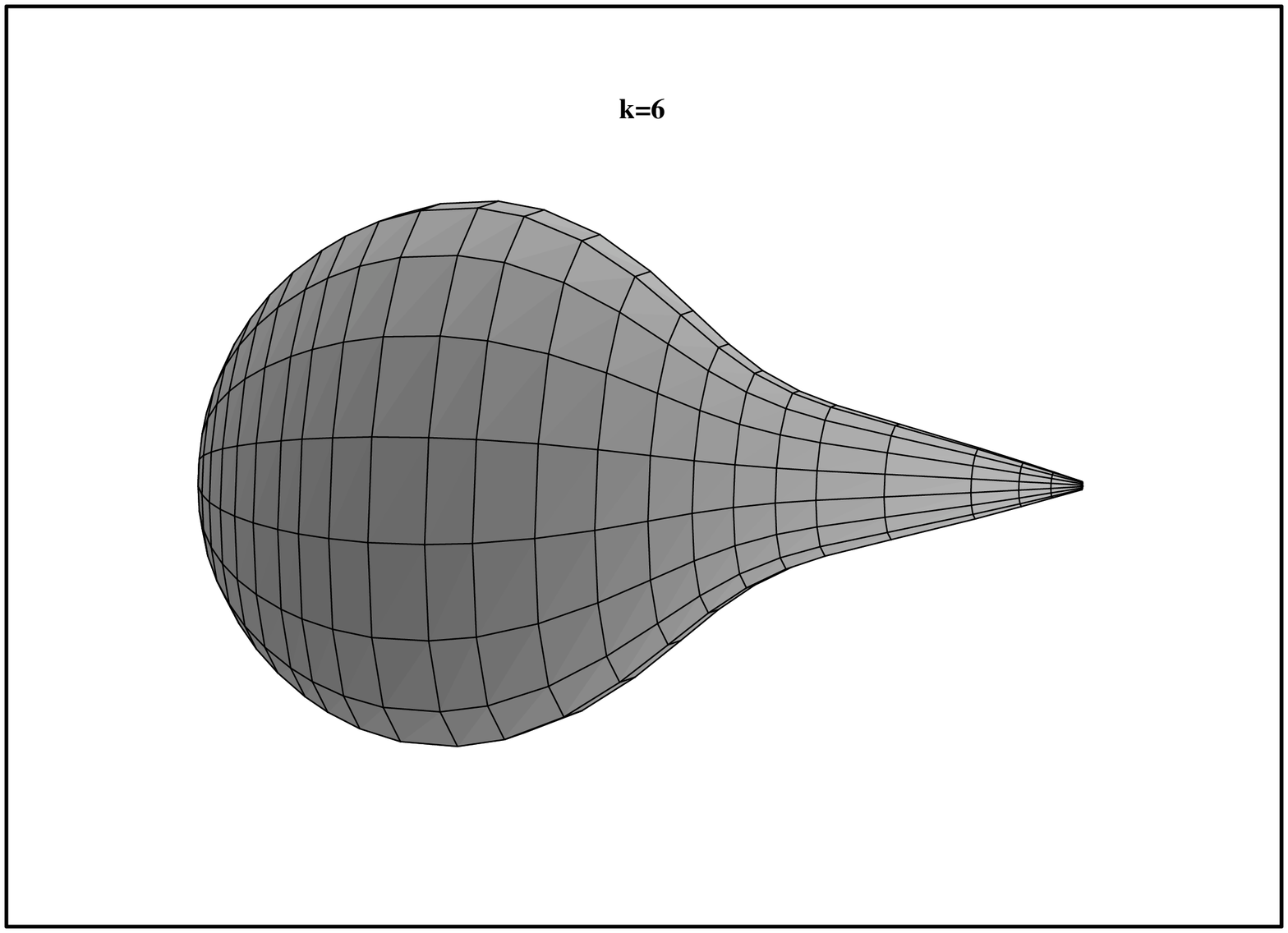}
\end{center}
\caption{Fixed point orbifold $2$-spheres in Hitchin metrics.}
\end{figure}

\end{document}